\documentclass[a4paper,12pt]{article}
\usepackage{graphicx,amssymb}
\usepackage{amssymb,amsmath}
\usepackage{hyperref}
\usepackage{amsfonts}
\usepackage{mathrsfs}
\usepackage{amscd}
\usepackage{amsthm}
\usepackage{color}
\usepackage{tikz-cd}
\usepackage[utf8]{inputenc}
\usepackage{longtable}
\numberwithin{equation}{section}

\usepackage{amssymb,amsmath,amsthm,authblk}
\usepackage{enumitem}
\newtheorem{theorem}{Theorem}
\newtheorem{corollary}{Corollary}
\newtheorem{remark}{Remark}
\newtheorem{lemma}{Lemma}

\begin{document}
\title{ \textbf{Spectra of eccentricity matrix of $H$-join of graphs}}
\baselineskip 16pt
\author{\large S. Balamoorthy and  T. Kavaskar 
\footnote{Corresponding Author. \\ E-Mail addresses: moorthybala545@gmail.com (S. Balamoorthy), t\_kavaskar@yahoo.com  (T. Kavaskar)}\\
{\small Department of Mathematics,\\ Central University of Tamil Nadu,\\ Thiruvarur 610 005, India.}}
\date{\today}
\maketitle      
\begin{abstract} Let $\varepsilon(G)$ be the eccentricity matrix of a graph $G$ and $Spec(\varepsilon(G))$ be the eccentricity spectrum of $G$. Let $H[G_1,G_2,\ldots, G_k]$ be the $H$-join of graphs $G_1,G_2,\ldots, G_k$ and let $H[G]$ be lexicographic product of $H$ and $G$. This paper finds the eccentricity matrix of a $H$-join of graphs. Using this result, we find (i) $Spec(\varepsilon(H[G]))$ in terms of $Spec(\varepsilon(H))$ if the radius $(rad(H))$ of $H$ is at least three; (ii) $Spec(\varepsilon(K_k[G_1,G_2,\ldots, G_k]))$ if $\Delta(G_i)\leq |V(G_i)|-2$ which generalises some of the  results in \cite{Mahato1}; (iii) $Spec(\varepsilon(H[G_1,G_2,\ldots, G_k]))$  if $rad(H)\geq 2$ and $G_i$ is complete whenever $e_H(i)=2$, which generalises some of the  results in  \cite{Mahato1} and \cite{Wang1}. Finally, we find the characteristic polynomial of $\varepsilon(K_{1,m}[G_0,G_1,\ldots, G_m])$ if $G_i$'s are regular. As a result, we deduce some of the results in \cite{Li}, \cite{Mahato1}, \cite{Patel} and \cite{Wang}.
\medskip
\end{abstract}

\section{Introduction}
We consider only finite simple undirected graphs. Let $G=(V(G), E(G))$ be a graph. For any $v \in V(G)$, $\text{deg}_G(v)$ denotes the degree of $v$ and $\Delta(G)$ denotes the maximum degree of $G$. For a subset $U$ of $V(G)$, $\langle U \rangle$ denotes the subgraph induced by $U$ of $G$. A graph $G$ is $k$-regular (where $k$ is a non-negative integer) if $\text{deg}_G(v)=k$, for all $v\in V(G)$. For a graph $G$, the \emph{complement} of $G$, denoted by $\overline{G}$, is the graph with $V(\overline{G})=V(G)$ and $E(\overline{G})=\{uv: uv\notin E(G)\}$.
%such that $u$ and $v$ are adjacent in $\overline{H}$ if and only if they are not adjacent in $H$. 
% the number of edges incident to vertex $v$ in graph $G$ is referred to as the degree of vertex $i$ in $H$, denoted by $\text{deg}(i)$
Let $H$ be a graph with $V(H)=\{1,2,\dots, k\}$ and let $\mathcal{F}= \{G_1,G_2,\ldots,G_k\}$ be a family of graphs. The \emph{$H$-join} operation of the graphs $G_1,G_2,\ldots,G_k$, denoted by  $H[G_1,G_2,\ldots,G_k]$, is obtained by replacing the vertex $i$ of $H$ by the graph $G_i$ for $1 \leq i \leq k$ and every vertex of $G_i$ is made adjacent with every vertex of $G_j$, whenever $i$ is adjacent to $j$ in $H$, that is, the vertex set $V(H[G_1,G_2,\ldots,G_k])= \displaystyle \bigcup_{\ell =1}^k V(G_{\ell})$ and edge set $E(H[G_1,G_2,\ldots,G_k]) =  \displaystyle \big(\bigcup_{\ell=1}^k E(G_{\ell}) \big) \cup \big(\bigcup_{ij \in E(H)} \{uv : u \in V(G_i), v \in V(G_j)\}\big)$ (see \cite{Saravanan}). If $G\cong G_i$, for $1\leq i\leq k$, then $H[G_1,G_2,\ldots, G_k]\cong H[G]$, the \emph{lexicographic product} of $H$ and $G$. If $H=K_2$, then $K_2[G_1,G_2]$ is the join, denoted by $G_1\vee G_2$, of $G_1$ and $G_2$.

%Consider two disjoint graphs $G_1$ and $G_2$ with distinguished vertices $v_1$ and $v_2$, respectively. The coalescenceof two disjoint graphs $G_1 \ast G_2 = (G_1, v_1) \ast (G_2, v_2)$ is the graph obtained by gluing $G_1$ and $G_2$ at the distinguished vertices (see \cite{Patel}). For the positive integers $m\geq 2$ and $n\geq 1$, the windmill graph $ W_{n+1}^m$ is the graph obtained by taking $m$ copies of complete graph $K_{n+1}$ with a vertex in common. Note that the windmill graph $W_{n+1}^m \cong \underbrace{K_{n+1} \ast \cdots \ast K_{n+1}}_{\textit{m term}}$ (See \cite{Patel}). One can see that, $W_{n+1}^m \cong K_{1,m}[K_1,K_{n_1},K_{n_2},\ldots, K_{n_m}]$, where $n_i =n$, for $1 \leq i \leq m$.
Consider two disjoint graphs $G_1$ and $G_2$ with distinguished vertices $v_1$ and $v_2$, respectively. The coalescence
of two disjoint graphs $G_1 \ast G_2 = (G_1, v_1) \ast (G_2, v_2)$ is the graph obtained by gluing $G_1$ and $G_2$ at the distinguished vertices (see \cite{Patel}). For the positive integers $t\geq 2$ and $n\geq 1$, the windmill graph $ W_{n+1}^t$ is the graph obtained by taking $t$ copies of complete graph $K_{n+1}$ with a vertex in common. Note that the windmill graph $W_{n+1}^t \cong \underbrace{K_{n+1} \ast \cdots \ast K_{n+1}}_{\textit{t term}}$ (See \cite{Patel}). 
One can see that, $W_{n+1}^t \cong K_{1,t}[K_1,K_{n_1},K_{n_2},\ldots, K_{n_t}]$, where $n_i =n$, for $1 \leq i \leq t$.

Given a family of graphs $H, G_1,G_2, \ldots,G_k$, with $V(H)=\{1,2,\ldots,k\}$, the \emph{generalized corona}, denoted $H \tilde{\circ} \wedge_{i=1}^{k} G_{i}$, is the graph obtained by taking one copy of graph $H$, $ G_1,G_2, \ldots,G_k$ and joining the $i^{th}$ vertex of $H$ to every vertex of $G_i$. If $G\cong G_i$, for $1\leq i\leq k$, then the graph $H \tilde{\circ} \wedge_{i=1}^{k} G_{i}$ is called the \emph{corona product} of $H$ and $G$, denoted by $H \circ G$.

%Given a family of graphs $H, G_1,G_2, \ldots,G_k$, with $V(H)=\{1,2,\ldots,k\}$, the \emph{generalized corona}, denoted $H \tilde{\circ} \wedge_{i=1}^{k} G_{i}$, is the graph obtained by taking one copy of graph $H$, $ G_1,G_2, \ldots,G_k$ and joining the $i^{th}$ vertex of $H$ to every vertex of $G_i$. If $G\cong G_i$, for $1\leq i\leq k$, then the graph $H \tilde{\circ} \wedge_{i=1}^{k} G_{i}$ is called the \emph{corona product} of $H$ and $G$, denoted by $H \circ G$.

%The join of two graphs $G_1$ and $G_2$ is denoted by $G_1 \vee G_2$, is obtained from $G_1 \cup G_2$ by joining every vertex of $G_1$ with every vertex of $G_2$.

For an  $n\times n$ complex matrix $B$ with eigenvalues $\lambda_1,\lambda_2, \ldots \lambda_n$, let $\phi(\lambda, B)= \det(\lambda I-B)$ denote  the characteristic polynomial of $B$ and $E(B)=\sum_{i=1}^n|\lambda_i|$ represent the energy of $B$. 
The spectral radius of $B$ is defined as 
$\rho(B)=$max$\{|\lambda|: \lambda $ is an eigenvalue of $B\}$.
%, denoted by $\rho(B)$, is the largest eigenvalue of $B$. 
%We call $\lambda$ is the spectral radius of $B$ if  $\lambda$ is the maximum of the absolute values of all other eigenvalues of $B$ and is denoted by $\rho(B)$. 
If $\lambda_1,\lambda_2,\ldots, \lambda_s$ are the only distinct eigenvalues of $B$, then the spectrum of $B$, denoted by $Spec(B)$, is defined as follows
\begin{align*}
Spec(B)= \begin{Bmatrix}
\lambda_1 & \lambda_2 & \ldots & \lambda_s\\
m_1 & m_2 & \ldots & m_s
\end{Bmatrix},   
\end{align*}
where $\lambda_i$  is the eigenvalue and $m_i$
is the algebraic multiplicity  for $1 \leq i \leq s$.
Note that, if  $B$ is a real symmetric matrix, then all the eigenvalues of $B$ are real. The inertia of $B$ is characterized by the triplet of integers $\big(N_+(B), \ N_0(B), \ N_{-}(B) \big)$, where $N_+(B)$, $N_0(B)$, and $N_{-}(B)$ denote the counts of positive, null, and negative eigenvalues of $B$, respectively. For a  block matrix $B_{ij}$ of $B$, denote $B_{ij}(r,s)$ the ${rs}^{ th}$ entry of $B_{ij}$.

For a graph $G$ with vertex set $V(G)=\{v_1,v_2,\ldots, v_n\}$, the adjacency matrix of $G$ is $A(G)= (b_{ij})$, where. 
\begin{align*}
 b_{ij}=\begin{cases}
\text{1} &\quad\text{if $v_iv_j\in E(G)$,} \\
\text{0} &\quad\text{otherwise.}
\end{cases}   
\end{align*}
%$A(G)$, is a matrix with rows and columns indexed by the vertices of $H$. Specifically, the $(i,j)$-entry of $A(H)$ is $1$ if the edge $ij$ is an element of $E(H)$, and $0$ otherwise, where $i$ and $j$ belong to the vertex set $V(H)$. 
%The adjacency matrix of the subgraph $U$ of $H$, induced by the vertex set $U$, is denoted as $A(\langle U \rangle)$. 

For a connected graph $G$ and $u,v\in V(G)$, the distance between $u$ and $v$ in $G$, denoted by $d_G(u,v)$, is the
length of a shortest path joining them in $G$ and the eccentricity of $u$, denoted by $e_G(u)$, is defined as $e_G(u)  = max\{d_G(u,v) : v\in V(G)\}$. The radius of $G$, denoted by $rad(G)$, is the minimum eccentricity of all vertices of $G$.

The eccentricity matrix, initially proposed by Randić \cite{Randic} as the DMax-matrix in 2013 and later renamed the Ecc-matrix by Wang et al. \cite{Wang} in 2018, is defined for a connected graph $G$ with $V(G)=\{v_1,v_2,\ldots, v_n\}$ is $\varepsilon(G) = (a_{ij})$, where 
\begin{align*}
 a_{ij}=\begin{cases}
\text{$d_G(v_i,v_j)$} &\quad\text{if $d_G(v_i,v_j) = min\{e_G(v_i),e_G(v_j)\}$} \\
\text{$0$} &\quad\text{otherwise.}
\end{cases}   
\end{align*}
%The eigenvalues of the eccentricity matrix of a graph $H$ are referred to as the $\varepsilon$-eigenvalues of $H$. Let $\lambda_1, \lambda_2,\ldots, \lambda_n$  represent the eigenvalues of the matrix $\varepsilon(G)$. As the matrix $\varepsilon(H)$ is symmetric, all the $\varepsilon$-eigenvalues of $H$ are real. If $\lambda_1,\lambda_2,\ldots, \lambda_r$ distinct eigenvalues of $\varepsilon(H)$ such that  $\lambda_1 > \lambda_2 >\ldots > \lambda_r$, then the $\varepsilon$-spectrum of $G$, denoted by $spec_{\varepsilon}(H)$, is defined as follows \begin{align*} spec_{\varepsilon}(H)= \begin{Bmatrix} \lambda_1 & \lambda_2 & \ldots & \lambda_n\\ m_1 & m_2 & \ldots & m_n \end{Bmatrix},    \end{align*} where $m_i$ is the algebraic multiplicity of the eigenvalue $\lambda_i$, for $1 \leq i \leq n$. The inertia of $\varepsilon(H)$ is characterized by the triplet of integers $\big(n_+(\varepsilon(H)), \ n_0(\varepsilon(H)), \ n_{-}(\varepsilon(H)) \big)$, where $n_+(\varepsilon(H))$, $n_0(\varepsilon(H))$, and $n_{-}(\varepsilon(H))$ denote the counts of positive, null, and negative eigenvalues of $\varepsilon(H)$, respectively.

Randić \cite{Randic} observed that the eccentricity matrix has potential applications not only in Chemical Graph Theory but also in Graph Theory, due to the importance of eccentricities in graph analysis. More recently, studies \cite{Mahato1, Wang1} have investigated additional spectral properties of the eccentricity matrix and explored the energy associated with it.

Wang et al. \cite{Wang1} introduced the concept of eccentricity energy, named as $\varepsilon$-energy, for a graph $G$ with $n$ vertices. This is formally defined as
\begin{align*}
    E(\varepsilon(G)) = \displaystyle \sum_{i=1}^n|\lambda_i|,
\end{align*}
where $\lambda_1, \lambda_2,\ldots, \lambda_n$ are the eigenvalues of $\varepsilon(G)$. The least eigenvalue of $\varepsilon(G)$ is denoted by $\xi(G)$. The characteristic polynomial, spectrum, energy, spectral radius and inertia of $G$ we mean that, the characteristic polynomial, spectrum, energy, spectral radius and inertia, respectively of its eccentricity matrix $\varepsilon(G)$ of $G$.

In \cite{Mahato2}, Mahato and Kannan examined the inertia of the eccentricity matrices of trees and characterized those trees that have exactly three distinct eccentricity eigenvalues. Additionally, He and Lu \cite{He} investigated trees with a given order and an odd diameter, focusing on maximizing the spectral radius of the eccentricity matrix. There are several papers on the eccentricity energy and eccentricity eigenvalues of graphs. See \cite{He}, \cite{Li}-\cite{Mahato2}, \cite{Sorgun}-\cite{Wang1}. %; we refer the reader to the reader
%\begin{definition}[\cite{Zhang}]
%Let $\varepsilon(H)$ denote the eccentricity matrix of a connected graph $H$. The inertia of $\varepsilon(H)$ is the triplet of integers $\big(n_+(\varepsilon(H)), \ n_0(\varepsilon(H)), \ n_{-}(\varepsilon(H)) \big)$, where $ n_+(\varepsilon(H)), \ n_0(\varepsilon(H))$ and $ n_{-}(\varepsilon(H))$ denote the number of positive, null and negative eigenvalues of $\varepsilon(H)$, respectively.   \end{definition}

We state the following results in \cite{Cvet, Cvetkovic, Roger, Saravanan, You} which will be used throughout this paper. % used in this paper.

\begin{lemma}[\cite{Cvet}]\label{BKP1008}
Let $A, B, C$ and $D$ be matrices such that $M=
\begin{bmatrix}
A & B\\
C & D
\end{bmatrix}
$. If $A$ is invertible, then  $\det(M) = \det(A) \det(D - CA^{-1}B).$
\end{lemma}

\begin{lemma}[\cite{Cvetkovic}]\label{BKP1010}
Let $N_1,N_2,N_3,N_4$ be $n\times n$ matrices. If $N_1$ is invertible and $N_1N_3 = N_3N_1$, then $\begin{vmatrix}
 N_1 & N_2\\
 N_3 & N_4
\end{vmatrix} = \begin{vmatrix}
    N_1N_4 - N_3N_2
\end{vmatrix}.$    
\end{lemma}

\begin{lemma}[\cite{Cvetkovic}]\label{BKP1013}
Let $G$ be an $r$-regular graph of order $n$. 
If $r, \lambda_2,\ldots,\lambda_n$ are eigenvalues of the adjacency matrix $A(G)$ of $G$, then  $n-r-1, -(\lambda_2 +1),\ldots,-(\lambda_n +1)$ are eigenvalues of the adjacency matrix $A(\overline{G})$ of $\overline{G}$.    
\end{lemma}

\begin{lemma}[\cite{Roger}]\label{BKP1002}
Let $A \in \mathbb{R}^{n \times n}$
and $B \in \mathbb{R}^{m \times m}$. If $Spec(A)$ and $Spec(B)$ denote the spectra of matrices $A$ and $B$ respectively, then the spectrum of their Kronecker product $A \otimes B$ is given by  $spec(A \otimes B) = \{\lambda \mu
:  \lambda \in Spec(A), \mu \in Spec(B) \}$.    
\end{lemma}

\begin{lemma}[\cite{Roger1}]\label{BKP1017}
 Two symmetric  matrices $A$ and $B$ of the same size have the same number of positive, negative and zero eigenvalues if and only if they are congruent (ie., $ B=SAS^{T}$, for some non-singular matrix S).    
\end{lemma}

\begin{theorem}[\cite{Saravanan}]\label{mainthm}
Let $M_i$ be a complex matrix of order $n_i$, and let $u_i$ and $v_i$ be arbitrary complex vectors of size $n_i \times 1$ for $1 \le i \le k$. Let $\sigma_{ij}$ be arbitrary complex numbers for $1 \le i,j \le k$ and $i \ne j$. For each $1 \le i \le k$, let $\phi(\lambda, M_i)=\det(\lambda I_{n_i}-M_i)$ be the characteristic polynomial of the matrix $M_i$ and $\Gamma_i(\lambda) = \Gamma_{M_i}(u_i,v_i;\lambda) = v_i^t (\lambda I - M_i)^{-1} u_i$. 
Considering the $k$-tuple $\bold M=(M_1, M_2, \dots, M_k)$, 2$k$-tuple $\bold u=(u_1,v_1, u_2,v_2 \dots, u_k,v_k)$  and ${k(k-1)}$-tuple$ \\
 \sigma =(\sigma_{12}, \sigma_{13} \dots, \sigma_{1k},\sigma_{21}, \sigma_{23}, \dots, \sigma_{2k}, \dots, \sigma_{k1}, \sigma_{k2}, \dots,  \sigma_{k k-1})$ the following matrices are defined: $$B(\bold M, \bold u, \sigma) := \begin{bmatrix}
	
	M_1 & \sigma_{12} u_1v_2^t & \cdots  & \sigma_{1k} u_1v_k^t \\
	\sigma_{21} u_2v_1^t & M_2 & \cdots  & \sigma_{2k} u_2v_k^t \\
	\vdots & \vdots & \ddots  & \vdots \\
	\sigma_{k1} u_kv_1^t & \sigma_{k2} u_kv_2^t & \cdots  &M_k
\end{bmatrix}$$ $$\text{ and }
 \widetilde{B}(\bold M, \bold u, \sigma) :=  \begin{bmatrix}
\frac{1}{\Gamma_1(\lambda)} & -\sigma_{12}  & \cdots & -\sigma_{1,k}  \\
-\sigma_{21}  & \frac{1}{\Gamma_2(\lambda)} & \cdots & -\sigma_{2,k} \\
\vdots & \vdots & \ddots & \vdots  \\
-\sigma_{k1} & -\sigma_{k2} & \cdots &\frac{1}{\Gamma_{k}(\lambda)}
\end{bmatrix}.$$

Then the characteristic polynomial of $B(\bold M, \bold u, \sigma)$ is given as
\begin{equation}\label{maineqn}
\det(\lambda I - B(\bold M, \bold u, \sigma)) = \big( \Pi_{i=1}^k \phi(\lambda, M_i) \Gamma_i(\lambda) \big) \det(\widetilde{B}(\bold M, \bold u, \sigma)) .
\end{equation} 
\end{theorem}

\begin{lemma}[\cite{Saravanan}]\label{evmain}
Let $M$ be a matrix of order $n$ with an eigenvector $u$ corresponding to the eigenvalue $\mu$ of $M$. Then $\Gamma_M(u, u;\lambda) = \dfrac{\Vert u \Vert^2}{\lambda-\mu}$.
\end{lemma}

\begin{lemma}[\cite{Wang1}]\label{BKP1012}
 Let $C_n$ be the cycle of length $n$. 
 \begin{enumerate}
    \item If $n=2t\,(t \geq 2)$, then $Spec(\varepsilon(C_{2t}))=\begin{Bmatrix}
   t & -t\\
   t & t
 \end{Bmatrix}$ 

\item If $n=2t+1\, (t \geq 1)$, then \\
$Spec(\varepsilon(C_{2t+1}))=
\begin{Bmatrix}
2t \cos(\frac{2\pi }{2t+1}) & 2t \cos(\frac{4 \pi }{2t+1}) & \ldots & 2t \cos({2\pi })\\
   1 & 1 & \ldots & 1
 \end{Bmatrix}$.
\end{enumerate}
\end{lemma}

\begin{theorem} [\cite{You}]\label{BKP1016}
Let  $M = (A_{ij})$  be a complex block matrix of order $n$ such that $A_{ij} = s_{ij}J_{n_i\times n_j}$ for $i\neq j$, and $A_{ii} =
s_{ii}J_{n_i\times n_i} + p_iI_{n_i}$. Then the equitable quotient matrix of $M$ is $Q = (q_{ij})$ with $q_{ij} = s_{ij}n_j$
if $i \neq j$, and $q_{ii} = s_{ii}n_i + p_i$. Moreover,
$Spec(M) = Spec(Q) \cup \{p_1^{[n_1 -1]}, \ldots,p_t^{[n_t - 1]}\}.$
\end{theorem}

\noindent The paper is organized as follows. 

In Section 2, we determine the eccentricity matrix of $H$-join of graphs. As a result, we find the inertia of $H[G_1,G_2,\ldots,G_k]$ if $rad(H) \geq 3$.

In Section 3, we find the characteristic polynomial, spectrum, energy and inertia of the eccentricity matrix of the following graphs. \\
(i) Lexicographic product of $H$ and $G$ in terms of $H$ if $rad(H)\geq 3$; \\
(ii) $H$-join of $G_1,G_2,\ldots, G_k$  if $rad(H)\geq 2$ and $G_i$ is complete whenever $e_H(i)=2$. As a result of (ii), we deduce some of the results in \cite{Mahato1} and \cite{Wang1}.\\
%(iii) $Spec(\varepsilon(H[K_m]))$ if $H$ is a connected graph with $\Delta(H)\leq |V(H)|-2$. As a result of (iii), we deduce some of the results in \cite{Mahato1};\\
(iii) $K_k[G_1,G_2,\ldots, G_k]$ if $\Delta(G_i)\leq |V(G_i)|-2$, which generalises a result in \cite{Mahato1}. As a result of (iii), we deduce some of the results in  \cite{Mahato1}. %that is,  $Spec(\varepsilon(K_{n_1,n_2,\ldots, n_k}))$.\\ %, which was proved in \cite{Mahato1}.\\ %Also, we find the spectrum of join of two graphs, which generalizes a result in \cite{Mahato1}. 

In the last section, %\textcolor{red}{we determine $\det(D-b(J-I))$, where $D$ is diagonal matrix and $b$ is a fixed constant with some more conditions (see Theorem \ref{D-J+I})}. Using this result, 
we find the characteristic polynomial of $\varepsilon(K_{1,k}[G_0,G_1,\ldots, G_k])$ if $G_i$'s are regular. As a result, we deduce some of the results in \cite{Mahato1}, \cite{Patel} and \cite{Wang}. 

%In the last section,  we determine $Spec(\varepsilon(f_{m,n}))$. As a consequence, we deduce a result in \cite{Patel} which is the  spectrum of coalescence of two cycles with same order. 

\section{Eccentricity matrix of $H$-join of graphs}
In this section, we begin with the following lemma. 
\begin{lemma}\label{BKP1004}
Let $H$ be a connected graph with $V(H)=\{1,2,\ldots,k\}$. Let $G_1,G_2,\ldots,G_k$ be a family of graphs with $|V(G_i)|=n_i$, for $i=1,2,\ldots,k$ and $G=H[G_1,G_2,\ldots,G_k]$. Then 
\begin{enumerate}[label = (\roman*)]
\item If $i,j \in V(H)$ with $i \neq j$ and $x,y \in V(G)$ with $x\neq y$, then \[d_G(x,y)= \begin{cases}
\text{$2$} &\quad\text{if $x,y \in V(G_i)\, \text{ with } xy \notin E(G_i)$}\\
\text{$d_H(i,j)$} &\quad\text{if $x \in V(G_i) \text{ and } y \in V(G_j)$.}  \end{cases}\]
\item If $i\in V(H)$, then $e_H(i) \leq e_G(x)$, for all $x \in V(G_i)$.

\item If $e_H(i) \geq 2$, then  
$e_H(i)=e_G(x)$, for all $x \in V(G_i)$.

\item  If $e_H(i) = 1$ and $x \in V(G_i)$, then   
\[e_G(x) = \begin{cases}
\text{$1$} &\quad\text{if $deg_{G_i}(x)= n_i -1,$} \\ \text{$2$} &\quad\text{otherwise.} \end{cases}\]
\end{enumerate}
\end{lemma}
\begin{proof}
\textbf{Proof of (i)} follows from the definition of $H$-join of graphs.

\noindent \textbf{Proof of (ii)}. 
For $i\in V(H)$, there exists $j\in V(H)\backslash \{i\}$ such that $e_H(i)=d_H(i,j)$. By (i),  $d_H(i,j) = d_G(x,y)$, for all $x \in V(G_i)$ and $y \in V(G_j)$. Hence $e_H(i) =d_H(i,j) = d_G(x,y)\leq e_G(x)$, for all $x \in V(G_i)$.\\
\noindent \textbf{Proofs of (iii) and (iv)} follow from (i) and (ii). 
%Suppose $e_H(i)= t \geq 2$. By (i), $e_H(i)=e_G(x)$ for all $x \in V(G_i)$. 
%then there exists  $j \in V(H)$ such that $d_H(i,j) =t \geq 2$. Suppose $d_G(x,y)>t$, where $x \in V(G_i)$ and $y \in V(G_{\ell})$, for some $\ell \in \{1,2,\ldots,k\}$, then by the definition of $H$-join $d_H(i,\ell) > t$, which leads to contradiction. Hence if $e_H(i) \geq 2$, then $e_H(i)=e_G(x)$ for all $x \in V(G_i)$. 
%\noindent Proof of (iv).
%If $e_H(v_i) = 1$ and $x \in V(G_i)$, then  \[e_G(x) = \begin{cases}
%\text{$1$} &\quad\text{if $e_{G_i}(x)= 1$} \\ \text{$2$} &\quad\text{otherwise.} \end{cases}\]
\end{proof}
Using Lemma \ref{BKP1004}, we prove the following result. 

\begin{theorem}\label{BKP1001}
Let $G_1,G_2,\ldots,G_k$ be a family of graphs with $|V(G_i)| = n_i$ and $k\geq 2$, $U_{i_1}=\{x \in V(G_i) : deg_{G_i}(x) = n_i-1\}$, $U_{i_2}= V(G_i) \setminus U_{i_1}$ and $|U_{i_t}| = m_{i_t}$, for $t =1, 2$, where $i \in \{1,2,\ldots,k\}$. Let $H$ be a connected graph with $V(H)=\{1,2,\ldots,k\}$ and $G=H[G_1,G_2,\ldots,G_k]$. If  
$\varepsilon(H)= (a_{ij})$,
%\begin{bmatrix}a_{11} & a_{12} & \ldots &a_{1k} \\a_{21} &a_{22} & \ldots &a_{2k}\\\vdots&   \vdots& \ \ddots &  \vdots&\\a_{k1} &a_{k2} & \ldots &a_{kk}\end{bmatrix}$,
then
$\varepsilon(G) = (A_{ij})$
%\begin{bmatrix}A_{11} & A_{12} & \ldots &A_{1k} \\A_{21} &A_{22} & \ldots &A_{2k}\\\vdots&   \vdots& \ \ddots &  \vdots&\\A_{k1} &A_{k2} & \ldots &A_{kk}\end{bmatrix},$
is the block matrix with the non-diagonal block matrix $A_{ij}$ ($i\neq j$) of $\varepsilon(G)$ is  
 \[A_{ij} = \begin{cases} \text{$a_{ij}J_{n_i\times n_j}$} &\quad\text{if $a_{ij} \geq 2$}\vspace{0.2 cm}\\ \text{$\begin{bmatrix}
J_{m_{i_1}\times m_{j_1}} & J_{m_{i_1}\times m_{j_2}}\\
J_{m_{i_2}\times m_{j_1}}& \textbf{0}_{m_{i_2}\times m_{j_2}}
\end{bmatrix}$}  &\quad\text{if $a_{ij}=1 \text{ and } e_H(i) =1 = e_H(j)$ }\vspace{0.2 cm}\\ 
\text{$\begin{bmatrix}
J_{m_{i_1}\times m_{j_1}} & J_{m_{i_1}\times m_{j_2}}\\
\textbf{0}_{m_{i_2}\times m_{j_1}}& \textbf{0}_{m_{i_2}\times m_{j_2}}
\end{bmatrix}$}  &\quad\text{if $a_{ij}=1, e_H(i) =1  \text{ and }  e_H(j) \neq 1$ }\vspace{0.2 cm}\\ 
\text{$\begin{bmatrix}
J_{m_{i_1}\times m_{j_1}} & \textbf{0}_{m_{i_1}\times m_{j_2}}\\
J_{m_{i_2}\times m_{j_1}}& \textbf{0}_{m_{i_2}\times m_{j_2}}
\end{bmatrix}$}  &\quad\text{if $a_{ij}=1,  e_H(i) \neq 1 \text{ and }  e_H(j) =1$}\vspace{0.2 cm}\\ 
\text{$\textbf{0}_{n_{i}\times n_{j}}$} &\quad\text{if $a_{ij}=0$}\\  \end{cases} \] and the diagonal block matrix $A_{ij}$ ($i= j$) of $\varepsilon(G)$ is  
\[A_{ii} =\begin{cases}
\text{$\begin{bmatrix}
{(J-I)}_{m_{i_1}\times m_{i_1}} & J_{m_{i_1}\times m_{i_2}}\\
J_{m_{i_2}\times m_{i_1}}& {2A(\overline{\langle U_{i_2}\rangle})}_{m_{i_2}\times m_{i_2}}
\end{bmatrix}$}  &\quad\text{if $e_H(i)=1$ }\vspace{0.2cm}\\ 
\text{$\begin{bmatrix}
\textbf{0}_{m_{i_1}\times m_{i_1}} & \textbf{0}_{m_{i_1}\times m_{i_2}}\\
\textbf{0}_{m_{i_2}\times m_{i_1}}& {2A(\overline{\langle U_{i_2}\rangle})}_{m_{i_2}\times m_{i_2}}
\end{bmatrix}$}  &\quad\text{if $e_H(i)=2$}\vspace{0.2cm}\\ 
\text{$\textbf{0}_{n_{i}\times n_{i}}$} &\quad\text{if $e_H(i) \geq 3$,}\\  \end{cases} \]
where $J_{n \times m}$ is the matrix whose entries are all $1$, $I_{n \times n}$ is an identity matrix and $A(\overline{\langle U_{i_2}\rangle})$ is the adjacency matrix of $\overline{\langle U_{i_2}\rangle}$.
%where $A(U_{i_2})$ is the adjacency matrix of the subgraph induced by $U_{i_2}$.
\end{theorem}
\begin{proof}
%Let $u,v \in V(G)$ with $u \neq v$ and $v_{i_t} \in V(G_i)$ for all $i = 1,2,\ldots,k$ and $t = 1,2,\ldots,n_i$. Then 
%\[d_G(u,v)= \begin{cases}\text{$2$} &\quad\text{if $uv \notin E(G_i) \text{ and } u,v \in V(G_i)$}\\\text{$d_H(v_i,v_j)$} &\quad\text{if $u \in V(G_i) \text{ and } v \in V(G_j)$}  %\end{cases}\]and $e_H(v_i) \leq e_G(v_{i_t})$ for all $t = 1,2,\ldots,n_i$. Observe that if $e_H(v_i) \geq 2$, then $e_H(v_i)=e_G(v_{i_t})$ for all $t = 1,2,\ldots, n_i$. If $e_H(v_i) = 1$, then \[e_G(v_{i_t}) = \begin{cases}\text{$1$} &\quad\text{if $deg_G(v_{i_t})= n_i -1$} \\ \text{$2$} &\quad\text{otherwise.} \end{cases}\]
Note that, if $U_{i_2}\neq \emptyset$, then $|U_{i_2}|\geq 2$.  For $i,j\in V(H)$. Let us first consider $i \neq j$.

\noindent{\textbf{Case 1.} $a_{ij} \geq 2$}.

Clearly, $d_H(i,j)= min\{e_H(i),e_H(j)\} \geq 2$. Then $e_H(i) \geq 2$ and $e_H(j) \geq 2$. By Lemma \ref{BKP1004}, $d_H(i,j) = d_G(x,y)$, $e_H(i) = e_G(x)$ and $e_H(j)=e_G(y)$, for all $x \in V(G_i)$ and $y \in V(G_j)$ and hence $d_H(i,j)=d_G(x,y)=min\{e_G(x),e_G(y)\}$. So, the block matrix $A_{ij}$ in $Ecc(G)$ is  $a_{ij}J_{n_i \times n_j}$.\\
\noindent{\textbf{ Case 2.}} $a_{ij}=1$.

Clearly, by the definition $a_{ij}$, $i$ is adjacent to $j$ and $e_H(i)=1$ or $e_H(j)=1$. 

\noindent{\textbf{Subcase 2.1.}} $e_H(i)=e_H(j) (=1)$.

For $x \in U_{i_2}$,  there exists $z\in V(G_i)$ such that $d_G(x,z)= 2$ and hence  $e_G(x)\geq 2$. Similarly, $e_G(y) \geq 2$, for $y\in U_{j_2}$. But $d_G(x,y)=d_G(i,j) = 1$. Hence $ A_{ij}(x,y) = 0$, for all $x \in  U_{i_2}$ and $y \in U_{j_2}$.

For $x \in U_{i_1}$ and $y \in V(G_j)$, $e_G(x)=1$ (by Lemma \ref{BKP1004}) %because, $deg_{G_i}(x)=n_i-1$ and $e_H(i)=1$) 
and
 $d_G(x,y) = 1 = min\{e_G(x), e_G(y)\}$. Therefore $A_{ij}(x,y) = 1$, for all $x \in  U_{i_1}$ and $y \in V(G_j)$.
 
By a similar argument, we see that  $A_{ij}(x,y) = 1$, for all $x \in V(G_i)$ and $y \in U_{j_1}$.
Hence \[ A_{ij}(x,y) = \begin{cases}
\text{$0$} &\quad\text{if $x \in U_{i_2} \text{ and } y \in U_{j_2}$}\\
\text{$1$} &\quad\text{otherwise.}
\end{cases} \]
Thus \[A_{ij}=\text{$\begin{bmatrix}
J_{m_{i_1}\times m_{j_1}} & J_{m_{i_1}\times m_{j_2}}\\
J_{m_{i_2}\times m_{j_1}}& 0_{m_{i_2}\times m_{j_2}}
\end{bmatrix}$}_{n_i \times n_j} \]
\noindent{\textbf{Subcase 2.2.}} $e_H(i)=1$ and $e_H(j) \neq 1$.

If $x \in U_{i_1}$ and $y \in V(G_j)$, then by Lemma \ref{BKP1004}, $e_G(x)=e_H(i)=1$ and $d_G(x,y)=1$ and therefore $ A_{ij}(x,y) = 1=min\{e_G(x),e_G(y)\}$, for $x \in U_{i_1}$ and $ y \in V(G_j)$.

If $x \in U_{i_2}$, $y \in V(G_j)$, then $e_G(x) = 2 $ and $e_G(y) = e_H(j)$ (as $e_H(j) \geq 2$), but $d_G(x,y)=1$, we have $A_{ij}(x,y)=0$, for $x \in U_{i_2}$ and $y\in V(G_j)$. Therefore, 
\[ A_{ij}(x,y) = \begin{cases}
\text{$1$} &\quad\text{if $x \in U_{i_1}\text{ and } y\in V(G_j),$}\\
\text{$0$} &\quad\text{otherwise.}
\end{cases} \]
Thus \[A_{ij}=\text{$\begin{bmatrix}
J_{m_{i_1}\times m_{j_1}} & J_{m_{i_1}\times m_{j_2}}\\
\textbf{0}_{m_{i_2}\times m_{j_1}}& \textbf{0}_{m_{i_2}\times m_{j_2}}
\end{bmatrix}$}_{n_i \times n_j} \]
\noindent{\textbf{Subcase 2.3.}}  $e_H(i) \neq 1$ and $e_H(j) = 1$.

By the similar proof in subcase $2.2$, we see that
\[ A_{ij}(x,y) = \begin{cases}
\text{$1$} &\quad\text{if $ x\in V(G_i) \text{ and }y \in U_{j_1},$}\\
\text{$0$} &\quad\text{otherwise.}
\end{cases} \]
and hence \[A_{ij}=\text{$\begin{bmatrix}
J_{m_{i_1}\times m_{j_1}} & \textbf{0}_{m_{i_1}\times m_{j_2}}\\
J_{m_{i_2}\times m_{j_1}}& \textbf{0}_{m_{i_2}\times m_{j_2}}
\end{bmatrix}$}_{n_i \times n_j}. \]
\noindent{\textbf{Case 3.}} $a_{ij} = 0$.

By the definition of $a_{ij}$, $d_H(i,j) < \min\{e_H(i),e_H(j)\}$, which implies that $e_H(i) \geq 2$ and $e_H(j) \geq 2$. By Lemma \ref{BKP1004}, $e_H(i)=e_G(x)$, $e_H(j)=e_G(y)$ for all $x\in V(G_i)$, $ y \in V(G_j)$ and $d_H(i,j)=d_G(x,y)$. Therefore, $A_{ij}(x,y)=0$, for all $x\in V(G_i)$, $y \in V(G_j)$ and hence $A_{ij} = \textbf{0}_{n_i \times n_j}$.\\
Therefore, we have obtained the block matrices $A_{ij}$'s of $\varepsilon(G)$ when $i\neq j$. 

Next let us consider $i=j$.
Here, let us first assume that  
%\noindent{\textbf{Case 4.}} 
$e_H(i) = 1$.

When $x \in U_{i_1}$ and $y \in V(G_i)$ with $x \neq y$. Then $e_G(x)= e_H(i)= 1$ 
and $d_G(x,y)=1$. Thus, $A_{ii}(x,y)=1$, for all $x \in U_{i_1},\ y \in V(G_i)$ and $x \neq y$.

When $x,y \in U_{i_2}$. Then $e_G(x)=e_G(y) = 2$. If $xy \in E(G_i)$, then $A_{ii}(x,y)=0$. %for all $x,y \in U_{i_2}$ and $xy \in E(G_i)$. 
If $xy \not \in E(G_i)$, then $d_G(x,y)=2$ and hence $A_{ii}(x,y) = 2$. %, for all $x,y \in U_{i_2}$ and $xy \not \in E(G_i)$.
So, we have \[ A_{ii}(x,y) 
 = \begin{cases}
 \text{$1$} &\quad\text{if $x \in U_{i_1} \text{ and } y \in V(G_i)$}\\ \text{$0$} &\quad\text{if $x,y \in U_{i_2} \text{ and } xy \in E(G_i)$}\\ \text{$2$} &\quad\text{if $x,y \in U_{i_2} \text{ and } xy \not \in E(G_i)$}\end{cases} \]
 and
thus \[ A_{ii}=  \text{$\begin{bmatrix}
{(J-I)}_{m_{i_1}\times m_{i_1}} & J_{m_{i_1}\times m_{i_2}}\\
J_{m_{i_2}\times m_{i_1}}& {2A(\overline{\langle U_{i_2}\rangle})}_{m_{i_2}\times m_{i_2}}
\end{bmatrix}$}_{n_i \times n_i}. \] 

Next, let us assume that $e_H(i) \geq 3$. As $d_G(x,y)\leq 2$, for $x,y \in V(G_i)$, we have  $A_{ii}=\textbf{0}_{n_i \times n_i}.$

Finally, let us assume that $e_H(i) = 2$. Then $e_G(x) = 2$, for all $x \in V(G_i)$(by Lemma \ref{BKP1004}). For $x,y\in V(G_i)$,
\[d_G(x,y) = \begin{cases}
 \text{$2$} &\quad\text{if $x,y \in U_{i_2} \text{ and } xy \not \in E(G_i)$}\\
\text{$1$} &\quad\text{otherwise,} 
\end{cases} \]

%As $d_G(x,y)=2$, for $x,y \in V(U_{i_2})$, 
we have  
 \[A_{ii}= \text{$\begin{bmatrix}
\textbf{0}_{m_{i_1}\times m_{i_1}} & \textbf{0}_{m_{i_1}\times m_{i_2}}\\
\textbf{0}_{m_{i_2}\times m_{i_1}}& {2A(\overline{\langle U_{i_2}\rangle})}_{m_{i_2}\times m_{i_2}}
\end{bmatrix}$}_{n_i \times n_i}. \]
Hence the proof.
\end{proof}
Using Theorem \ref{BKP1001}, we obtain the inertia of $H$-join of graphs as follows. 

\begin{corollary}\label{BKP1019}
Let $H$ be a connected graph with $V(H) = \{1,2,\ldots,k\}$, $\varepsilon(H)=(a_{ij})$ and $rad(H) \geq 3$. Let $G_1, G_2,\ldots, G_k$ be a family of graphs with $ |V(G_i)| = n_i$, $n = \sum_{i=1}^kn_i$ and $G=H[G_1, G_2,\ldots, G_k]$. Then $\phi(\lambda,\, \varepsilon(G)) = \lambda^{n-k}\, \det(\lambda I  - Q)$, where\\    
\[ Q = (q_{ij}) = \begin{cases}
\text{$0$} &\quad\text{if $i=j$}\\
\text{$a_{ij}n_j $} &\quad\text{otherwise},
\end{cases} \]
%%0 & a_{12}n_2 & a_{13}n_3 & \ldots & a_{1k}n_k \\
%a_{21}n_1& 0 & a_{12}n_2 &  \ldots & a_{2k}n_k\\
%a_{31}n_1 & a_{32}n_2 & 0 &   \ldots & a_{3k}n_k  \\
%\vdots & \vdots & \vdots & \ddots & \vdots\\
%a_{k1}n_1 & a_{k2}n_2 & a_{k3}n_3 &  \ldots & 0
%\end{bmatrix};$$
$ N_+(\varepsilon(G)) =  N_+(\varepsilon(H))$;
$ N_0(\varepsilon(G)) = n-k +  N_0(\varepsilon(H))$ and 
$ N_-(\varepsilon(G)) =   N_-(\varepsilon(H))$.
\end{corollary}
\begin{proof}
%\textbf{Proof of (i)} %Let $\varepsilon(H)=(a_{ij})$. 
As $e_H(i)\geq 3$ for $1\leq i\leq k$, we have  $a_{ij}=0$ or $a_{ij}\geq 3$, for $1 \leq i < j \leq k$. 
By Theorem \ref{BKP1001}, the block matrix $A_{ij}$ in $\varepsilon(G)$ is $a_{ij}J_{n_i \times n_j}$. Let $M=\varepsilon(G)=(A_{ij})$,  $s_{ij}=a_{ij} $ and $p_i = 0$, for $i,j \in \{1,2,\ldots,k\}$.
Then by Theorem \ref{BKP1016}, the equitable quotient matrix of $\varepsilon(G)$ is $Q = (q_{ij})$ and   
$Spec(\varepsilon(G)) = Spec(Q) \cup \begin{Bmatrix}
   0\\n-k 
\end{Bmatrix}$ and $\phi(\lambda,\, \varepsilon(G)) = \lambda^{n-k}\, \det(\lambda I  - Q)$.
%\textbf{Proof of (ii), (iii) and (iv)}.
Note that, $\varepsilon(H)$ and $Q$ are congruent (this is because $Q=D^{\frac{1}{2}}\varepsilon(H)D^{\frac{1}{2}}$, where $D$ is a diagonal matrix of order $k$ and whose diagonal elements are $n_1,n_2,\ldots,n_k$). Hence by Lemma \ref{BKP1017}, the result follows. 
%Let us take $A = \varepsilon(H)$, $B=Q$ and $S = D^{\frac{1}{2}}$, where $D$ is a diagonal matrix of order $k$ and whose diagonal elements are $n_1,n_2,\ldots,n_k$. 
%By Lemma \ref{BKP1017},  the inertia of $A$ and  $B$ are same. Thus we get the required result.
\end{proof}

The following result is an immediate consequence of Corollary \ref{BKP1019}.

\begin{corollary}
Let $H$ be a connected graph with $k$ vertices and $rad(H) \geq 2$ and let $m$ be a positive integer. If $Spec(\varepsilon(H)) = \begin{Bmatrix}
\lambda_1 & \lambda_2 &\ldots &\lambda_s\\
r_1 & r_2 &\ldots & r_s 
\end{Bmatrix}$, then \\ $Spec(\varepsilon(H[K_m])) = \begin{Bmatrix}
0 & m\lambda_1 & \ldots & m\lambda_s\\
k(m-1) & r_1 & \ldots & r_s
\end{Bmatrix}.$     
\end{corollary}
\begin{proof}
By Corollary \ref{BKP1019}, $Spec(\varepsilon(H[K_m])) = Spec(Q) \cup \begin{Bmatrix}
   0\\k(m-1) 
\end{Bmatrix},$  where $Q = m\varepsilon(H)$. Hence the result follows.   
\end{proof}

\section{Spectrum of $H$-join of some graphs}

In this section, we find the characteristic polynomial, spectrum, spectral radius and inertia of $H$-join of some graphs. 
%we first obtain the characteristic polynomial and inertia of $H[G_1,G_2,\\ \ldots,G_k]$ whenever $e_H(v_i) = 2, G_i = K_{n_i}$ and $rad(H) \geq 2$. Additionally, we find the spectrum of $P_4[G_1, G_2, G_3, G_4]$, where $G_2 = G_3 = k_1$. As a result, we deduce some results presented in \cite{Mahato1, Wang1} . Finally, we find the spectrum of $\varepsilon(H[G_1, G_2,\ldots,G_k])$ in terms of the spectrum of $\varepsilon(H)$ alone, but not depends on the spectrum of $\varepsilon(G_i)$, for $i = 1,2,\ldots,k$. 

First, let us find the spectrum of $\varepsilon(H[G])$ in terms of the spectrum of $\varepsilon(H)$ alone, but not depending on the spectrums of $\varepsilon(G_1), \varepsilon(G_2),\ldots, \varepsilon(G_k)$.

\begin{theorem}\label{BKP1003}
Let $H$ be a connected graph with vertex set $V(H)= \{1,2,\ldots,k\},$  $rad(H)\geq 3$ and $Spec(\varepsilon(H)) = \begin{Bmatrix}
\lambda_1 & \lambda_2 &\ldots &\lambda_s\\
r_1 & r_2 &\ldots & r_s 
\end{Bmatrix}$. Let $G_1,G_2,\ldots,G_k$ be a family of graph with $|V(G_i)|=m$, for $i = 1,2,\ldots,k$. Then \\
$Spec(\varepsilon(H[G_1,G_2,\ldots,G_k])) = \begin{Bmatrix}
0 & m\lambda_1 & \ldots & m\lambda_s\\
k(m-1) & r_1 & \ldots & r_s
\end{Bmatrix}$. In particular, if $G$ is a any graph with $|V(G)|=m$, then 
$Spec(\varepsilon(H[G])) = \begin{Bmatrix}
0 & m\lambda_1 & \ldots & m\lambda_s\\
k(m-1) & r_1 & \ldots & r_s
\end{Bmatrix}$.
\end{theorem}
\begin{proof}
%Since $e_H(v) \geq 3$, for all $v \in V(H)$, the entries of the eccentric matrix $(Ecc(H) = (a_{ij}))$ for $H$ are either equal to zero or greater than or equal to $3$.  
Let $\varepsilon(H)=(a_{ij})$. As $e_H(i)\geq 3$ for $1\leq i\leq k$, we have $a_{ij}=0$ or $a_{ij}\geq 3$, for $1 \leq i < j \leq k$.
%\begin{bmatrix}
%a_{11} & a_{12} & \ldots &a_{1n} \\
%a_{21} &a_{22} & \ldots &a_{2n}\\
 %\vdots&   \vdots& \ \ddots &  \vdots&\\
%a_{n1} &a_{n2} & \ldots &a_{nn}
%\end{bmatrix}$. 
By Theorem \ref{BKP1001}, the block matrix $A_{ij}$ in $\varepsilon(H[G_1,G_2,\ldots,G_k])$ is $a_{ij}J_{m \times m}$. Hence
$\varepsilon(H[G_1,G_2,\ldots,G_k])= \varepsilon(H) \otimes J_{m\times m}$. %Let us assume $Spec(Ecc(H))=\{\lambda_i : i = 1,2,\ldots,n\}$.
As $Spec(J_{m \times m}) =\begin{Bmatrix}
0 & m\\
(m-1) & 1 
\end{Bmatrix}$, by the Lemma \ref{BKP1002},  $Spec(\varepsilon(H[G_1,G_2,\ldots,G_k])) = 
\begin{Bmatrix}
0 & m\lambda_1 & \ldots & m\lambda_s\\
k(m-1) & r_1 & \ldots & r_s
\end{Bmatrix} $.   
\end{proof}
The following corollary is an immediate consequence of Theorem \ref{BKP1003}. 

\begin{corollary}
Let $C_n$ be a cycle on $n \geq 6$ vertices and $G_1, G_2,\ldots, G_k$ be a family of graphs on $m$ vertices.
\begin{enumerate}
    \item If $n=2t$, then $Spec(\varepsilon(C_n[G_1,G_2,\ldots,G_k]))= \begin{Bmatrix}
0 & mt & -mt\\
n(m-1)& t & t
\end{Bmatrix}.$  

\item If $n=2t+1$, then 
$Spec(\varepsilon(C_n[G_1,G_2,\ldots,G_k])) \\ = 
\begin{Bmatrix}
  0 & m(2t \cos\frac{2\pi }{2t+1}) & m(2t \cos\frac{4\pi }{2t+1})& \ldots & m(2t \cos 2\pi)\\   
  n(m-1) & 1 & 1 & \ldots & 1
\end{Bmatrix}.$
\end{enumerate}
\end{corollary}
\begin{proof}
 By Lemma \ref{BKP1012} and Theorem \ref{BKP1003}, the result follows. 
\end{proof}

In Theorem \ref{BKP1003}, we have assumed that $rad(H)\geq 3$. 
When we relax this condition, it very difficult to find the characteristic polynomial of $\varepsilon(H[G_1,G_2,\ldots,G_k])$ and its spectrum. %is more complicated. 
So, we consider a graph $H$ with the eccentricity of some vertices of $H$ is at most 2 (that is, $rad(H)\leq 2$) and some special graphs $G_1,G_2,\ldots,G_k$ in the following results. 

\begin{theorem}\label{BKP1005}
%Let $K_m$ be a complete graph with $m$ vertices and
Let $G_1,G_2,\ldots, G_k$ be a family of graphs with $|V(G_i)|=n_i$ and $\Delta(G_i) \leq n_i-2$, for $1\leq i \leq k$ and let $G= K_k[G_1,G_2,\ldots,G_k]$.  Then $\phi(\lambda, \varepsilon(G))=\prod_{i=1}^k\phi(\lambda, 2A(\overline{G_i}))$ and $\displaystyle Spec(\varepsilon(G))=\cup_{j=1}^k Spec({2A(\overline{G_j}}))$, where $A(\overline{G_j})$ is the adjacency matrix of $\overline{G_j}$. %the set of all eigenvalues of $\varepsilon(G)$ is $\displaystyle\cup_{j=1}^m \{2\lambda:\ \lambda\in Spec({\overline{G_j}})\}$.
\end{theorem}
\begin{proof} As $\Delta(G_i) \leq n_i -2$, for $1 \leq i \leq k$, by Lemma \ref{BKP1004} we have $e_G(x) =2,$ for all $x \in V(G)$. By Theorem \ref{BKP1001},  the non-diagonal block matrix $A_{ij}$ of $\varepsilon(G)$ is zero and the diagonal block matrix $A_{ii}=2A(\overline{G_i})$. 
Hence the result follows. % $Spec_{\varepsilon}(G)= \displaystyle\cup_{j=1}^m \{2\lambda:\ \lambda\in Spec({\overline{G_j}})\}$.
\end{proof}

%The following results are immediate consequences of the above 
Using Theorem \ref{BKP1005}, we deduce the following results in \cite{Mahato1}.
\begin{corollary} [Theorem 4.6, \cite{Mahato1}]
Let $G = K_{n_1,n_2,\ldots,n_k}$ be the complete $k$-partite graph such that $\sum_{i=1}^kn_i=n;\ n_i \geq 2$ and $2 \leq k \leq n-1$. Then \\$Spec(\varepsilon(G))=\begin{Bmatrix}
-2 & 2(n_1-1) & 2(n_2 -1) & \ldots & 2(n_k-1)\\
n-k & 1 & 1 & \ldots & 1
\end{Bmatrix}.$  
\end{corollary}
\begin{proof}
We know that $K_{n_1,n_2,\ldots,n_k}=K_k\big[\, \overline{K_{n_1}},\overline{K_{n_2}},\ldots, \overline{K_{n_k}}\, \big]$. By Theorem $\ref{BKP1005}$, the result follows, because $Spec(A(K_{n_i}))=\begin{Bmatrix}
-1 & (n_i-1) \\
n_i-1 & 1 
\end{Bmatrix}$, for $1\leq i\leq k$. 
\end{proof}
%We next deduce the following result in \cite{Mahato1} using Theorem \ref{BKP1005}. 

\begin{corollary}[Theorem 4.7, \cite{Mahato1}]
 Let $G_1$ and $G_2$ be two graphs with $n_1$ and $n_2$ vertices, respectively. If $G_1$ and $G_2$ are non-complete graphs (that is, $\Delta(G_i)\leq n_i-2$, for $i=1,2$), then $\varepsilon(G_1 \vee G_2)= \begin{bmatrix}
  2A(\overline{G_1})& 0\\
  0 & 2A(\overline{G_2})
 \end{bmatrix}$ and $\phi\big(\lambda, \varepsilon(G_1\vee G_2)\big)=\phi(\lambda, 2A(\overline{G_1}))\phi(\lambda, 2A(\overline{G_2}))$ and $\displaystyle Spec(\varepsilon(G_1\vee G_2))= Spec({2A(\overline{G_1}}))\cup Spec({2A(\overline{G_2}}))$. 
 \end{corollary}
 
We now recall the following result in \cite{Mahato1}.

\begin{theorem}[Theorem 4.1, \cite{Mahato1}]\label{BKP1009} Let $K_k$ be the complete graph on $k$ vertices, and let $G$ be any connected graph on $n$ vertices. Then\\
$Spec(\varepsilon(K_n \circ G)) =\begin{Bmatrix}
 0 &  \frac{-3n\pm \sqrt{9n^2+16n}}{2} & (k-1)\frac{-3n\pm \sqrt{9n^2+16n}}{2} \\
 k(n-1) & k-1 & 1
\end{Bmatrix}.$    
\end{theorem}
We now generalize the Theorem \ref{BKP1009} as follows. 
\begin{theorem}
Let $H = K_k \tilde{\circ} \wedge_{i=1}^{k} G_{i}$, where $G_i$ is any graph with $|G_i| =n$, for  $i=1,2,\ldots,k$. Then $Spec(\varepsilon(H)) = \begin{Bmatrix}
 0 &  \frac{-3n\pm \sqrt{9n^2+16n}}{2} & (k-1)\frac{-3n\pm \sqrt{9n^2+16n}}{2} \\
 k(n-1) & k-1 & 1
\end{Bmatrix}.$    
\end{theorem}
\begin{proof} 
%For $1\leq i\leq k$, let $V(G_i)=\{v_{ij}: 1\leq j\leq n\}$ and let $V(K_k)=\{v_i: 1\leq i\leq k\}$. We now label the vertices of $H$ as follows. $v_1,v_2,\ldots, v_k,v_{11},v_{21},\ldots,v_{n1}\ldots,v_{1i},v_{2i},\ldots,\\ v_{ni},\ldots,v_{1m}, v_{2m},\ldots,v_{nm}$. With this labeling of $H$, the eccentricity matrix of $H$ is given by  $\varepsilon(H) = A \otimes (J-I)_{k \times k}$, where $A = \begin{bmatrix} 0 & 2J_{1\times n}\\ 2J_{n\times 1}& 3J_n\end{bmatrix}$. Note that, the spectrum of the $A$ and $(J-I)$ are well known.  So the result follows from Lemma \ref{BKP1008}. 
The proof is same as in Theorem \ref{BKP1009} (See \cite{Mahato1}).% with the same label of the vertices of $H$.
\end{proof}

We next prove that the following result is much more general than the results given $\cite{Mahato1}$ and $\cite{Wang1}$.

\begin{theorem}\label{BKP1014}
Let $H$ be a graph with $V(H) = \{1,2,\ldots,k\}$, $\varepsilon(H)=(a_{ij})$ and $rad(H)\geq 2$. Let $G_1,G_2,\ldots,G_k$ be a family of graphs with $|V(G_i)|=n_i$, $G=H[G_1,G_2,\ldots,G_k]$ and $n = \sum_{i=1}^k n_i$.  Whenever $e_H(i)=2$, $G_i= K_{n_i}$. Then %, for $i\in \{1,2,\ldots,k\}$ and 
\begin{enumerate}[label = (\roman*)]
\item $Spec(\varepsilon(G)) = Spec(Q) \cup \begin{Bmatrix}
   0\\n-k 
\end{Bmatrix}$ and $\phi(\lambda,\, \varepsilon(G)) = \lambda^{n-k}\, \det(\lambda I  - Q)$, where  \[ Q = (q_{ij}) = \begin{cases}
\text{$0$} &\quad\text{if $i=j$}\\
\text{$a_{ij}n_j $} &\quad\text{otherwise}.
\end{cases} \]
\item $N_+(\varepsilon(G)) = N_+(\varepsilon(H))$, $N_0(\varepsilon(G)) = n-k + N_0(\varepsilon(H))$ and $N_-(\varepsilon(G)) = N_-(\varepsilon(H))$.
\end{enumerate}
\end{theorem}
\begin{proof}
As $e_H(i) \geq 2$, for $1 \leq i \leq k$, we have $a_{ij} = 0$ or $a_{ij} \geq 2$. If $e_{H}(i) = 2$, then $G_i=K_{n_i}$ and hence $A_{ii}= \textbf{0}_{n_i \times n_i}$ and $A_{ij}=a_{ij}J_{n_i \times n_j}$ ($i \neq j)$. Similarly, we note that if $e_{H}(i) \geq 3$, then $A_{ij}=a_{ij}J_{n_i \times n_j}$ for all $i,j \in \{1,2,\ldots,k\}$.
Hence $\varepsilon(G) = a_{i,j}J_{n_i \times n_j}$. Let $M=\varepsilon(G)$,  $s_{ij}=a_{ij} $ and $p_i = 0$, for $i,j \in \{1,2,\ldots,k\}$.
Then by Theorem \ref{BKP1016} and Lemma \ref{BKP1017}, we get the required result.
%$$Spec(\varepsilon(G)) =Spec(Q) \cup \begin{Bmatrix}0\\n-k \end{Bmatrix}, $$where $Q= \begin{bmatrix} 0 & a_{12}n_2 & a_{13}n_3 & \ldots & a_{1k}n_k \\a_{21}n_1& 0 & a_{12}n_2 &  \ldots & a_{2k}n_k\\a_{31}n_1 & a_{32}n_2 & 0 &   \ldots & a_{3k}n_k  \\\vdots & \vdots & \vdots & \ddots & \vdots\\a_{k1}n_1 & a_{k2}n_2 & a_{k3}n_3 &  \ldots & 0\end{bmatrix}.$
\end{proof}

Using Theorem \ref{BKP1014}, we prove the following result. % is a consequence of Theorem \ref{BKP1014}. 
\begin{theorem}\label{BKP1015}
Let $P_4$ be a path on $4$ vertices  and $G=P_4[G_1,G_2,G_3,G_4]$ with $|G_i|=n_i$, where $G_2\cong G_3\cong K_1$. Then     
$$Spec(\varepsilon(G)) = \begin{Bmatrix}
\sqrt{\frac{\alpha + \sqrt{\beta}}{2}} & \sqrt{\frac{\alpha - \sqrt{\beta}}{2}} & 0 & - \sqrt{\frac{\alpha - \sqrt{\beta}}{2}} & -\sqrt{\frac{\alpha + \sqrt{\beta}}{2}}\\
1 & 1 & n-4 & 1 & 1
\end{Bmatrix},$$
where $\alpha = 4n_1 + 4n_4 + 9n_1n_4$ and $\beta = (4n_1 + 4n_4 + 9n_1n_4)^2 -64n_1n_4$.
Hence 
\begin{align*}
E(\varepsilon(G)) &= \sqrt{2\big[4n_1+4n_2+9n_1n_4+ \sqrt{(4n_1 + 4n_4 + 9n_1n_4)^2 -64n_1n_4\big]}}\\  & \hspace{1cm}+ \sqrt{2\big[4n_1+4n_2+9n_1n_4 - \sqrt{(4n_1 + 4n_4 + 9n_1n_4)^2 -64n_1n_4\big]}}.  \end{align*}
\end{theorem}
\begin{proof}
Replace $H=P_4$ and $G_2 = G_3 = K_1$ in Theorem $\ref{BKP1014}$, we have    
$\phi(\lambda, \varepsilon(G)) = \lambda^{n-4} \det(\lambda I - Q)$, where $n = 2+n_1 + n_4$ and $ 
Q = \begin{bmatrix} 
0 & 0 & 2 & 3n_4\\
0 & 0 & 0 &  2n_4\\
2n_1 & 0 & 0 &   0   \\
3n_1 & 2 & 0 &  0
\end{bmatrix}.$ 
Now, 
\begin{align*}
 det(\lambda I - Q)& = \begin{vmatrix}
\lambda & 0 & 2 & 3n_4\\
0 &\lambda & 0 &  2n_4\\
2n_1 & 0 & \lambda &   0   \\
3n_1 & 2 & 0 &  \lambda     
 \end{vmatrix} \\
 & = \begin{vmatrix}
   \lambda I_{2 \times 2}& B\\
   C & \lambda I_{2 \times 2}
 \end{vmatrix}, 
\end{align*}
where $B= \begin{bmatrix}
 2 & 3n_4\\
 0 & 2n_4
\end{bmatrix}$ and $C = \begin{bmatrix}
 2n_1 & 0\\
 3n_1 & 2
\end{bmatrix}.$
Let us take $N_1= \lambda I_{2 \times 2}  = N_4$,\, $N_2= B$ and $N_3 = C$. 
Applying Lemma \ref{BKP1010}, we have
$$ det(\lambda I - Q) = (\lambda^4 - (4n_1 + 4n_4 + 9n_1n_4)\lambda^2 + 16n_1n_4).$$ 
Thus 
\begin{equation}\label{BKP1eq2.13}
\phi(\lambda, \varepsilon(G)) = \lambda^{n-4}\big( \lambda^4 - (4n_1 + 4n_4 + 9n_1n_4)\lambda^2 + 16n_1n_4\big).    
\end{equation}
After solving the characteristic polynomial, we get
$$Spec(\varepsilon(G)) = \begin{Bmatrix}
\sqrt{\frac{\alpha + \sqrt{\beta}}{2}} & \sqrt{\frac{\alpha - \sqrt{\beta}}{2}} & 0 & - \sqrt{\frac{\alpha - \sqrt{\beta}}{2}} & -\sqrt{\frac{\alpha + \sqrt{\beta}}{2}}\\
1 & 1 & n-4 & 1 & 1
\end{Bmatrix},$$
where $\alpha = 4n_1 + 4n_4 + 9n_1n_4$ and $\beta = (4n_1 + 4n_4 + 9n_1n_4)^2 -64n_1n_4$.
\end{proof}
Using Theorem \ref{BKP1015}, we deduce the following results in %The following results are proved in 
\cite{Wang1} and \cite{Mahato1}. 
%are consequence of Corollary \ref{BKP1015}

\begin{corollary}[Lemma 2.7, \cite{Wang1}]
For $a,b \geq 1$, the double star $S_{a,b}$  is the graph consisting of the union of two stars
$K_{1,a}$ and $K_{1,b}$ together with an edge joining their centers, %the eccentricity spectrum of  double star $S_{a,b}$ is given by  
$$Spec(\varepsilon(S_{a,b})) = \begin{Bmatrix}
\sqrt{\frac{\alpha + \sqrt{\beta}}{2}} & \sqrt{\frac{\alpha - \sqrt{\beta}}{2}} & 0 & - \sqrt{\frac{\alpha - \sqrt{\beta}}{2}} & -\sqrt{\frac{\alpha + \sqrt{\beta}}{2}}\\
1 & 1 &a+b-2 & 1 & 1
\end{Bmatrix},$$
where $\alpha = 4a + 4b + 9ab$ and $\beta = (4a + 4b + 9ab)^2 -64ab$. 
\end{corollary}
\begin{proof}
Clearly, we see that, $S_{a,b} = P_4[\overline{K_a},K_1,K_1,\overline{K_b}]$ and by Theorem \ref{BKP1015}, the result follows.    
\end{proof}

\begin{corollary}[Theorem 2.3, \cite{Wang1}] For $a,b \geq 1$,
\begin{align*}
E(\varepsilon(S_{a,b})) &= \sqrt{2\big(4a+4b+9ab+ \sqrt{(4a + 4b + 9ab)^2 -64ab\big)}}\\  & \hspace{1cm}+ \sqrt{2\big(4a+4b+9ab - \sqrt{(4a + 4b + 9ab)^2 -64ab\big)}}.  \end{align*}    
\end{corollary}

\begin{corollary}[Theorem 4.4, \cite{Mahato1}]
The $n$-barbell graph $B_{n,n}$ is
the graph obtained by connecting two copies of complete graph $K_n$ by a bridge (cut edge), $$Spec(\varepsilon(B_{n,n})) = \begin{Bmatrix}
 0 & \frac{3(n-1) \pm \sqrt{9n^2-2n-7}}{2} & \frac{-3(n-1) \pm \sqrt{9n^2-2n-7}}{2}\\ 
  2(n-2) & 1 & 1
\end{Bmatrix}$$     
\end{corollary}
\begin{proof}
One can see that, $B_{n,n} = P_4[K_{n-1},K_1,K_1,K_{n-1}]$. By Equation \ref{BKP1eq2.13}, we have $$\phi(\lambda, \varepsilon(B_{n,n}))= \lambda^{2n-4}\big( \lambda^4 - (8(n-1) + 9(n-1)^2)\lambda^2 + 16(n-1)^2\big)$$
After solving the characteristic polynomial of $B_{n,n}$, we get $\lambda^{2n-4} = 0$ or
\begin{align*}
\lambda^{2} & = \frac{8(n-1) + 9(n-1)^2 \pm \sqrt{(8(n-1) + 9(n-1)^2)^2- 64(n-1)^2}}{2}\\
& = \frac{(9n^2-10n+1) \pm 3(n-1)\sqrt{9n^2-2n-7}}{2}\\
& = \Bigg(\frac{3(n-1) \pm \sqrt{9n^2 -2n -7}}{2} \Bigg)^2
\end{align*}
Hence the result follows.
\end{proof}

\section{Spectrum of $K_{1,m}$-join of regular graphs} In this section, we first recall the following well-known result in \cite{Bapat}.

\begin{theorem}[Lemma 9.3, \label{D-J+I}\cite{Bapat}]
Let $A$ be an $n\times n$ matrix and $J$ be the $n\times n$ all one matrix. Then det$(A+J)=$det $A$+cof $A$.
\end{theorem}

Next, let us prove the following result using Theorems \ref{mainthm}, \ref{BKP1001} and \ref{D-J+I} and
Lemma \ref{evmain}. %, we prove the following result.

\begin{theorem}\label{BKP1006}
Let $K_{1,m}$ be a star graph with $V(K_{1,m}) = \{v_0,v_1,\ldots,v_m\}$, and $deg_{K_{1,m}}(v_0)= m$. 
Let $\overline{G_i}$ be a $k_i$-regular graph with $n_i$ vertices and $A(\overline{G_i})$ be its adjacency matrix, for $0 \leq i \leq m$ and let $H = K_{1,m}[G_0,G_1,\ldots,G_m]$.
\begin{enumerate}[label = (\roman*)]
        \item If $G_0$ is complete, then $$\phi(\lambda, \varepsilon(H)) = \lambda^{n_0 -1} \Bigg(\displaystyle \prod_{i=1}^{m} \frac{\phi(\lambda, 2A(\overline{G_i}))}{\lambda-2k_i} \Bigg)  \prod_{i=1}^m\big(\lambda - 2( k_i -n_i)\big) \Big[\lambda - (2\lambda + n_0) \sum_{i=1}^m \frac{n_i}{\lambda - 2 (k_i - n_i)} \Big]. $$
\item If $G_0$ is not complete, then $$\phi(\lambda, \varepsilon(H)) = \Bigg(\frac{\displaystyle \prod_{i=0}^{m} \phi(\lambda, 2A(\overline{G_i}))}{\displaystyle \prod_{i=1}^{m}\lambda-2k_i} \Bigg)  \prod_{i=1}^m\big(\lambda - 2( k_i -n_i)\big) \Big[1- 2\sum_{i=1}^m \frac{n_i}{\lambda - 2 (k_i - n_i)} \Big]. $$
%Let $spec_\varepsilon(H) = \begin{Bmatrix}\frac{b+\sqrt{b^2+ 4kn}}{2}  & 0 & \frac{b -\sqrt{b^2+ 4kn}}{2} & r-2n \\1 & k(n-1) & 1 & k-1\end{Bmatrix}$ and $E_{\varepsilon}(H) = b+ |r-2n|(k-1)$, where $b = r+ 2n(k-1)$.  
\end{enumerate}
\end{theorem}  
\begin{proof}
For $1 \leq i < j\leq m$,   $e_{K_{1,m}}(v_i)= d_{K_{1,m}}(v_i, v_j)= e_{K_{1,m}}(v_j)=2$, we have $a_{ij} =2$. Also, $e_{K_{1,m}}(v_0)= 1= d_{K_{1,m}}(v_0,v_i)$, for $i\in \{1,2,\ldots, m\}$. The eccentricity matrix of $K_{1,m}$ is indexed by the vertices of $v_0,v_1,\ldots,v_m$ and therefore 
\begin{equation} \label{BKP1eq2.4}
 \varepsilon(K_{1,m}) = (a_{ij}) = \begin{cases}
\text{$0$} &\quad\text{if $i=j$}\\
\text{$1$} &\quad\text{if $i \neq j$ \text{ and } $i= 1$ \text{ or } $j = 1$}\\
\text{$2$} 
&\quad\text{otherwise},
\end{cases} 
\end{equation}
%$a_{0,i}= a_{i,0} = 1$. Thus 
where $i,j \in \{0,1,2\ldots,m\}$. For $0 \leq i \leq m$,  $u_i=J_{n_i\times1}$  is an eigenvector corresponding to the eigenvalue $k_i$ of     $A(\overline{G_i})$  (as 
$\overline{G_i}$ is $k_i$-regular).
For $0 \leq i \leq m$, define $v_i= u_i$, $M_i = 2A(\overline{G_i})$ and  
\begin{equation}\label{BKP1eq2.1}
 \phi(\lambda, M_i) = \det(\lambda I - M_i). 
\end{equation}
%where $M_i = 2A(\overline{G_i})$. So, $M = (2A(\overline{G_0}),2A(\overline{G_1 }),\ldots,2A(\overline{G_m}))$ in Theorem ,By Lemma $\Gamma_i(\lambda) = $
%where $M_i = 2A(\overline{G_i})$. So, 
%in Theorem \ref{mainthm}. 
So, by Lemma \ref{evmain},
\begin{equation}\label{BKP1eq2.2}
\Gamma_i(\lambda) = \frac{{\lVert u_i \rVert}^2}{\lambda -2k_i}= \frac{n_i}{\lambda -2k_i}.   
\end{equation}
(as, $u_i$ is an eigenvector of $A(\overline{G_i})$ corresponding to the eigenvalue $k_i$).\\
Let $M = (2A(\overline{G_0}),2A(\overline{G_1 }),\ldots,2A(\overline{G_m}))$.
 %The eccentricity matrix of $K_{1,m}$ is indexed by the vertices $\{v_0,v_1,\ldots, v_m\}$ and its eccentricity matrix is

The eccentricity matrix of $H$ is indexed by the vertices of $G_0, G_1, \ldots, G_m$, respectively. \\
\noindent\textbf{Case 1.} $G_0$ is  complete, that is, $k_0=0$.\\
As $\overline{G_i}$ is regular, we have $U_{i_2}= \emptyset$ or $U_{i_2}= V(G_i)$. More precisely, if $\overline{G_i}$ is complete, then $U_{i_2}= \emptyset$ and if $\overline{G_i}$ is not complete, then $U_{i_2}= V(G_i)$. Note that \[ e_H(x) =  \begin{cases}
  \text{$1$} &\quad\text{if $x \in V(G_0)$}\\
  \text{2}  &\quad\text{if $x \in V(H)\setminus V(G_0)$}.
\end{cases} \]   So, by Equation  \ref{BKP1eq2.4} and Theorem \ref{BKP1001},
\begin{align*}
\varepsilon(H)= \begin{bmatrix}
\textbf{0}_{n_0\times n_0} & J_{n_0\times n_1} & J_{n_0\times n_2}& \ldots & J_{n_0\times n_m}\\
J_{n_1\times n_0}& 2A(\overline{G_1}) & 2 J_{n_1\times n_2}& \ldots & 2J_{n_1\times n_m}\\
J_{n_2\times n_0}& 2J_{n_2\times n_1}& 2A(\overline{G_2}) & \ldots & 2J_{n_2\times n_m}\\
\vdots & \vdots & \vdots & \ddots & \vdots\\
J_{n_m\times n_0}& 2J_{n_m\times n_1}& 2J_{n_m\times n_2} & \ldots & 2A(\overline{G_m})
\end{bmatrix}.    
\end{align*}
For $i \neq j$, choose,\[\sigma_{ij} = \begin{cases}
 \text{$1$} &\quad\text{if $i=1$ \text{ or } $j=1$ }\\
 \text{$2$} &\quad\text{otherwise.}
\end{cases}\]  
One can check that, $B(\bold M,\bold u,\sigma)=\varepsilon(H)$.\\
By Theorem \ref{mainthm}, Equations \ref{BKP1eq2.1} and 
 \ref{BKP1eq2.2} and as $G_0$ is complete, we have $\phi(\lambda, \varepsilon(H))=$
\begin{equation}\label{BKP1eq2.6}
 \bigg(\prod_{i=0}^{m} \phi\big(\lambda, 2A(\overline{G_i})\big)\Gamma_i(\lambda)\bigg)\det( \tilde{\varepsilon}(H))   = n_0\lambda^{n_0 -1} \bigg(\displaystyle \prod_{i=1}^{m} n_i\, \frac{\phi(\lambda, 2A(\overline{G_i}))}{\lambda-2k_i} \bigg)\,  det(\tilde{\varepsilon}(H)),
\end{equation}   
where
\begin{equation}\label{BKP1eq2.5}
 \hspace{-2.2cm}\tilde{\varepsilon}(H) = 
\begin{bmatrix}
\frac{\lambda}{n_0}& -1 & -1 & \ldots &-1\\
-1 & \frac{\lambda-2k_1}{n_1} & -2 &  \ldots & -2 \\
-1 & -2 & \frac{\lambda-2k_2}{n_2} &   \ldots & -2& \\
\vdots & \vdots & \vdots & \ddots & \vdots\\
-1 & -2 & -2 &  \ldots &  \frac{\lambda-2k_
m}{n_m}
\end{bmatrix}, (\text{as }k_0=0)   \end{equation}

Next, we have to find $\det(\tilde{\varepsilon}(H))$. Let $r_i = \frac{\lambda-2k_
i}{n_i}, \text{ for } 0\leq i\leq m$. Then
\begin{equation*}
\tilde{\varepsilon}(H) = \begin{bmatrix}
r_0& -1 & -1 & \ldots & -1\\
-1 & r_1 & -2 &  \ldots & -2 \\
-1 & -2 & r_2 &   \ldots & -2 & \\
\vdots & \vdots & \vdots & \ddots & \vdots\\
-1 & -2 & -2 &  \ldots &  r_m
\end{bmatrix} 
\end{equation*}
Now, let us consider the following sub matrices of $\tilde{\varepsilon}(H)$.\\
Let $A=r_0I_{1\times 1}$, $B= -J_{1\times m}= 
\begin{bmatrix}  -1 & -1 & \ldots &-1 \end{bmatrix}_{1 \times m}$, 
$C=B^t$ and 
\begin{equation*}
 D = 
\begin{bmatrix}
 r_1 & -2 &  \ldots & -2 \\
 -2 & r_2 &   \ldots & -2 \\
  \vdots & \vdots & \ddots & \vdots\\
 -2 & -2 &  \ldots &  r_m
\end{bmatrix}.
\end{equation*}
Applying Lemma \ref{BKP1008}, we have 
\begin{align*} 
\det(\tilde{\varepsilon}(H))=det(A)det(D-CA^{-1}B)= r_0 \begin{vmatrix}
    r_1 - \frac{1}{r_0} &  -(2 + \frac{1}{r_0}) & \ldots & -(2 + \frac{1}{r_0}) \vspace{0.2cm}\\
    -(2 + \frac{1}{r_0}) & r_2 - \frac{1}{r_0} & \ldots &-(2 + \frac{1}{r_0})\\
    \vdots & \vdots & \ddots &\vdots\\
    -(2 + \frac{1}{r_0}) & -(2 + \frac{1}{r_0}) & \ldots & r_m - \frac{1}{r_0}
\end{vmatrix} 
\end{align*}
Let $b=2+\frac{1}{r_0}$ and $M =(a_{ij})$ with \[ a_{ij} = \begin{cases}
\text{$r_i-\frac{1}{r_0}$} &\quad\text{if $i=j$ \text{ and } $1 \leq i \leq m$}\\
\text{$0$} &\quad\text{if } i\neq j.
\end{cases} \]
Then \begin{align*}  det(\tilde{\varepsilon}(H))& =r_0\ det[M-b(J-I)] \\ & =r_0 \big( (-b)^m \det[A + J] \big), \text{ where } A= \Big[\Big(\frac{-1}{b}\Big) M - I\Big]. 
\end{align*}
 %where $ A= [\big(\frac{-1}{b}\big) M - I], b=2+\frac{1}{r_0}$ and $M =(m_{ij})$ with \[ m_{ij} = \begin{cases} \text{$r_i-\frac{1}{r_0}$} &\quad\text{if $i=j$ \text{ and } $1 \leq i \leq m$}\\ \text{$0$} &\quad\text{if } i\neq j. \end{cases} \]
By Theorem \ref{D-J+I}, 
\begin{align*}
det(\tilde{\varepsilon}(H)) & = r_0 \Big( (-b)^m \Big[\Big(\frac{-1}{b}\Big)^{m} \prod_{i=1}^m(a_{ii}+b) + \Big(\frac{-1}{b}\Big)^{m-1} \sum_{i=1}^m \frac{\prod_{j=1}^m(a_{jj}+b)}{a_{ii}+b}\Big]\Big)\\
& = r_0 \prod_{i=1}^m(a_{ii}+b) \Big[1 - b \sum_{i=1}^m \frac{1}{a_{ii}+b} \Big]\\
& = r_0 \prod_{i=1}^m(r_i + 2) \Big[1 - (2 + \frac{1}{r_0}) \sum_{i=1}^m \frac{1}{r_i + 2} \Big].
\end{align*}
Therefore, $det(\tilde{\varepsilon}(H)) =\displaystyle \prod_{i=0}^{m} \frac{1}{n_i}\prod_{i=1}^m\big(\lambda - 2( k_i -n_i)\big) \Big[\lambda - (2\lambda + n_0) \sum_{i=1}^m \frac{n_i}{\lambda - 2 (k_i - n_i)} \Big]. $\\
Hence by Equation \ref{BKP1eq2.6}, the result follows.
 \\
\noindent\textbf{Case 2.} $G_0$ is $k_0$-regular and not complete.\\
Note that $e_H(x) = 2$, for all $x \in V(H)$.
Since $G_i$'s are regular, by Equation \ref{BKP1eq2.4} and Theorem \ref{BKP1001}, we have
\begin{align*}
\varepsilon(H)= \begin{bmatrix}
2A(\overline{G_0}) & \textbf{0}_{n_0\times n_1} & \textbf{0}_{n_0\times n_2}& \ldots & \textbf{0}_{n_0\times n_m}\\
\textbf{0}_{n_1\times n_0}& 2A(\overline{G_1}) & 2 J_{n_1\times n_2}& \ldots & 2J_{n_1\times n_m}\\
\textbf{0}_{n_2\times n_0}& 2J_{n_2\times n_1}& 2A(\overline{G_2}) & \ldots & 2J_{n_2\times n_m}\\
\vdots & \vdots & \vdots & \ddots & \vdots\\
\textbf{0}_{n_m\times n_0}& 2J_{n_m\times n_1}& 2J_{n_m\times n_2} & \ldots & 2A(\overline{G_m})
\end{bmatrix}  
\end{align*}
Choose,\[\sigma_{ij} = \begin{cases}
 \text{$0$} &\quad\text{if $i=1$ \text{ or } $j=1$}\\
 \text{$2$} &\quad\text{otherwise.}
\end{cases}\] 
One can check that, $B(M,u,\sigma)=\varepsilon(H)$.\\
By Theorem \ref{mainthm}, Equation \ref{BKP1eq2.1} and 
 \ref{BKP1eq2.2}, we get 
\begin{equation*}
\hspace{-4.4cm}\phi(\lambda, \varepsilon(H))  = \bigg(\prod_{i=0}^{m} \phi\big(\lambda, 2A(\overline{G_i})\big)\Gamma_i(\lambda)\bigg) \det(  \tilde{\varepsilon}(H))     
\end{equation*}
\begin{equation}\label{BKP1eq2.8}
\hspace{1.9cm}=\displaystyle \prod_{i=0}^{m} n_i \frac{\phi(\lambda, 2A(\overline{G_i}))}{(\lambda-2k_i)}\, \det(\tilde{\varepsilon}(H)) \,(\text{as $G_0$ is not complete}). 
\end{equation}
where 
\begin{equation*}
\tilde{\varepsilon}(H) = \begin{bmatrix}
\frac{\lambda-2k_
0}{n_0}& 0 & 0 & \ldots &0\\
0 & \frac{\lambda-2k_
1}{n_1} & -2 &  \ldots & -2 \\
0 & -2 &\frac{\lambda-2k_
2}{n_2}&   \ldots & -2 & \\
\vdots & \vdots & \vdots & \ddots & \vdots\\
0 & -2 & -2 &  \ldots &  \frac{\lambda-2k_
m}{n_m}
\end{bmatrix} .
\end{equation*}
For $0\leq i\leq m$, let $r_i=\frac{\lambda-2k_
i}{n_i}$. Then

\begin{equation}\label{BKP1eq2.7}
\tilde{\varepsilon}(H) = \begin{bmatrix}
r_0& 0 & 0 & \ldots &0\\
0 & r_1 & -2 &  \ldots & -2 \\
0 & -2 & r_2 &   \ldots & -2 & \\
\vdots & \vdots & \vdots & \ddots & \vdots\\
0 & -2 & -2 &  \ldots &  r_m
\end{bmatrix} 
\end{equation}
Therefore,   
\begin{align*}
 det(\tilde{\varepsilon}(H))& =r_0\ det[M-2(J-I)] \\ & =r_0 \big( (-2)^m \det[A + J] \big), 
\end{align*}
 where $ A= (\frac{-1}{2}M - I)$ and 
$M=(a_{ij})$ with \[ a_{ij} = \begin{cases}
\text{$r_i$} &\quad\text{if $i=j$ \text{ and } $1 \leq i \leq m$}\\
\text{$0$} &\quad\text{if } i\neq j.
\end{cases} \]
%Let us assume that $r_i= \frac{\lambda-2k_i}{n_i}$, for $i= 0,1,2,\ldots,m$. \\
By Theorem \ref{D-J+I}, \begin{align*}
det(\tilde{\varepsilon}(H)) & = r_0 \Big( (-2)^m \Big[\Big(\frac{-1}{2}\Big)^{m} \prod_{i=1}^m(a_{ii}+2) + \Big(\frac{-1}{2}\Big)^{m-1} \sum_{i=1}^m \frac{\prod_{j=1}^m(a_{jj}+2)}{a_{ii}+2}\Big]\Big)\\
& = r_0 \prod_{i=1}^m(a_{ii}+2) \Big[1 - 2 \sum_{i=1}^m \frac{1}{a_{ii}+2} \Big]\\
& = r_0 \prod_{i=1}^m(r_i + 2) \Big[1 - 2  \sum_{i=1}^m \frac{1}{r_i + 2} \Big].
\end{align*}
Therefore, $det(\tilde{\varepsilon}(H)) = \displaystyle\frac{\lambda -2k_0}{n_0}\prod_{i=1}^m\Big(\frac{\lambda - 2( k_i -n_i)}{n_i}\Big) \Big[1- 2\sum_{i=1}^m \frac{n_i}{\lambda - 2 (k_i - n_i)} \Big]$. \\
Hence by Equation \ref{BKP1eq2.8}, the result follows. 
\end{proof}

Using Theorem \ref{BKP1006}, we deduce the following results in 
\cite{Mahato1, Patel} and \cite{Wang}.

\begin{corollary}[Theorem 3.2, \cite{Patel}]
If $a,b \geq 3$, then the eccentricity spectrum of $K_a 
 \ast K_b$ is given by \\ $Spec(\varepsilon(K_a 
 \ast K_b))=
 \begin{Bmatrix}
2\sqrt{R}\ \cos(\frac{\theta}{3}) & 0 &  2\sqrt{R}\ \cos(\frac{\theta+ 4\pi}{3}) & 2\sqrt{R}\ \cos(\frac{\theta+ 2\pi}{3})\\
1 & a+b-4 & 1 & 1
 \end{Bmatrix}$ and  $E(\varepsilon(K_a 
 \ast K_b))= 4\sqrt{R}\ \cos(\frac{\theta}{3})=2 \rho(\varepsilon(K_a 
 \ast K_b)),$ where $R=\frac{1}{3}(4ab -3(a+b) +2)$ and $\theta =\cos^{-1}\big(\frac{2(a-1)(b-1)}{\sqrt{R^3}}\big)$. \end{corollary}
\begin{proof}
Clearly, $K_a 
 \ast K_b= K_{1,2}[K_1,K_{a-1},K_{b-1}]$. Then by Theorem $\ref{BKP1006}(i)$, 
$r_0 = \lambda$, $r_1 = \frac{\lambda}{a-1}$, $r_2 = \frac{\lambda}{b-1}$ and hence 
\begin{equation}\label{BKP1eq2.9}
\phi(\lambda, \varepsilon((K_a 
 \ast K_b))= \lambda^{a+b-4}\big(\lambda^3 - (4ab -3(a+b)+2)\lambda -4(a-1)(b-1)\big).  
\end{equation}
After computing the Equation \ref{BKP1eq2.9}, we get the required results. %  The same proof is given in Theorem 3.2 (see \cite{Patel}), and we obtained the result.  
\end{proof}
\begin{corollary}[Lemma 3.6, \cite{Wang}]\label{BKP1018}
Let $G$ be a $r$-regular graph with  $n$ vertices and diameter two. If $r, \lambda_2, \ldots, \lambda_n$ are the eigenvalues of the adjacency matrix $A(G)$ of $G$, then
$$Spec(K_1 \vee G) = \begin{Bmatrix}
(n-r-1)\pm \sqrt{(n-r-1)^2+n} & -2(\lambda_2+1)& \ldots & -2(\lambda_n + 1)\\
1 & 1 & \ldots & 1
\end{Bmatrix}.$$
\end{corollary}
\begin{proof}
Clearly, $K_1 \vee G = K_{1,1}[K_1, G]$. Then by Theorem \ref{BKP1006}(i), $n_0 = 1$, $n_1 = n$, $k_0 =0$ and $k_1= (n-r-1)$. Hence
$$\phi(\lambda, \varepsilon(K_1 \vee G)) = \frac{\phi(\lambda, 2A(\overline{G_1}))}{\lambda - 2(n-r-1)}(\lambda^2 - 2(n-r-1)\lambda -n).$$
The result follows from Lemma \ref{BKP1013}.
\end{proof}

\begin{remark}
In Corollary \ref{BKP1018}, the diameter $2$ is not mandatory, but the diameter at least $2$ is mandatory.   
\end{remark}

\begin{corollary}[Theorem 4.3, \cite{Mahato1}]
Let $W_{n+1}$ be the wheel graph on $n+1$ vertices, with $n \geq 4$. Then  
$$Spec(W_{n+1}) = \begin{Bmatrix}
(n-3)\pm \sqrt{(n-3)^2+n} & -2(\lambda_2+1)& \ldots & -2(\lambda_n + 1)\\
1 & 1 & \ldots & 1
\end{Bmatrix},$$ where $\lambda_2, \lambda_3, \ldots, \lambda_n$ are the eigen value of $A(G)$ other than $2$.
\end{corollary}
\begin{proof}
Clearly, $W_{n+1} = K_{1,1}[K_1, C_n]$. Then by Theorem \ref{BKP1006}(i), $n_0 = 1$, $n_1 = n$, $k_0 =0$ and $k_1= (n-3)$. Hence,
$$\phi(\lambda, \varepsilon(W_{n+1})) = \frac{\phi(\lambda, 2A(\overline{G_1}))}{\lambda - 2(n-3)}(\lambda^2 - 2(n-3)\lambda -n).$$
By Lemma \ref{BKP1013}, the result follows.
\end{proof}

\begin{corollary}[Lemma 3.1, \cite{Li}]
For $n \geq 5$, denote $S_{n,3}$ graph with vertex set $V(S_{n,3})=\{v_1,v_2,\ldots,v_n\}$ and the edge set $E(S_{n,3}) = \{v_1v_2, v_2v_3, v_1v_3\} \cup \{v_3v_i : 4 \leq i \leq n\}$, 
$Spec(\varepsilon(S_{n,3})) =\begin{Bmatrix}
-2 & 0 & t_1 & t_2 & t_3 \\
n - 4&  1& 1 & 1 &1
\end{Bmatrix}$
where $t_1 > t_2 > t_3$ are the roots of the equation $\lambda^3+(8-2n)\lambda^2+(25-9n)\lambda+(8-4n) = 0$. Further, $-4 < t_3 = \xi(S_{n,3}) < -3$ \end{corollary}
\begin{proof}
One can see that, $S_{n,3} = K_{1,2}[K_1,K_2,\overline{K_{n-3}}]$. Then by Theorem \ref{BKP1006}(i) $n_0 = 1$, $n_1 = 2$, $n_2 = n-3$, $k_0 = k_1 =0$ and $k_2= (n-4)$. Hence,$$\phi(\lambda, \varepsilon(S_{n,3})) = (\lambda -2)^{n-4} (\lambda^4 + (8 -2n)\lambda^3 + (25-9n)\lambda^2 + (8 -4n)\lambda).$$   
After solving the characteristic polynomial, we get the required results.
\end{proof}

Also using Theorem \ref{BKP1006}, we prove the following interesting result which is a much more general result than the result given in (Theorem 3.3, \cite{Patel}). 

\begin{theorem}\label{BKP1007}
Let $H = K_{1,m}[G_0,G_1,\ldots,G_m]$, where $\overline{G_0}$ is a $r$-regular graph with $\ell$ vertices and $\overline{G_i}$ is a $k$-regular graph with $n$ vertices for $1 \leq i \leq m$. %If there exists $t\in \{1,2,\ldots, k\}$ such that $G_i$ is not complete for $1\leq i\leq t$ and $G_i$ is  complete for $t+1\leq i\leq k$. 
Then 

\begin{enumerate}[label = (\roman*)]
\item If $G_0$  is  complete, then $\phi(\lambda I, \varepsilon(H)) = \\  \bigg(\displaystyle\prod_{i=1}^m  \frac{\phi\big(\lambda, 2A(\overline{G_i})\big)}{(\lambda-2k)} \bigg) \lambda^{\ell-1}
\big[(\lambda -2(k-n))^{m-1}(\lambda^2 -2(k+n(m-1))\lambda -\ell mn)\big]$,

\begin{table}[h!]
\renewcommand*{\arraystretch}{1.5}
    \centering
 \begin{tabular}[c]{|c |c| c| }
 \hline
   & $G_1,\ldots,G_m$ are complete & $G_1,\ldots,G_m$ are not  complete \\
\hline
$N_+(\varepsilon(H))$ & 1 & $\sum_{i=1}^m N_+\big(\varepsilon(A(\overline{G_i})) + 1-m$\\
\hline
$N_0(\varepsilon(H))$ & $m(n-1)+ \ell -1$ & $\sum_{i=1}^m N_0\big(\varepsilon(A(\overline{G_i}))\big) +\ell - 1$\\
\hline
$N_-(\varepsilon(H))$ & $m$ & $\sum_{i=1}^m N_-\big(\varepsilon(A(\overline{G_i}))\big)+m$\\
\hline
%& & \\
$\rho(H)$ &$n(m-1)+\sqrt{(n(m-1))^2+\ell mn}$ &  $k+n(m-1)+\sqrt{(k+n(m-1))^2+\ell mn}$\\
%& & \\
\hline
$E(\varepsilon(H))$ & $2\rho(H)$ & $\sum_{i=1}^m E\big(\varepsilon(2A(\overline{G_i}))\big)+2\rho(H) -4km$\\
\hline
\end{tabular}
\end{table}
%$N_+(\varepsilon(H)) = \displaystyle \sum_{i=1}^m N_+\big(\varepsilon(A(\overline{G_i})\big) -(m- 1)$,\\$N_0(\varepsilon(H)) =\displaystyle \sum_{i=1}^m$$N_0\big(\varepsilon(A(\overline{G_i})\big) + \ell -1$, \\ $N_{-}(\varepsilon(H)) = \displaystyle\sum_{i=1}^m N_-\big(\varepsilon(A(\overline{G_i})\big) +m$ and its eccentricity energy  is 
%$E(\varepsilon(H)) = \displaystyle \bigg( \sum_{i=1}^m E\big(2A(\overline{G_i})\big)- 2mk \bigg) + 2(\rho(H) -mk)$. %where  $n_0\big((\varepsilon(A(\overline{G_i}))\big) = b_i$ and $n_{-}\big(\varepsilon(A(\overline{G_i}))\big) = c_i$.
\item If $G_0$ is not complete but $k_0$-regular, then $\phi(\lambda I, \varepsilon(H)) = \\ \displaystyle  \frac{\prod_{i=0}^m\phi(\lambda, 2A(\overline{G_i}))}{(\lambda-2k)^m} \  \{\lambda - 2(k+n(m-1))\}\{\lambda -2(k-n)\}^{m-1}$,  
\begin{table}[h!]
 \renewcommand*{\arraystretch}{1.5}
    \centering
 \begin{tabular}[c]{|c |c| c| }
 \hline
   & $G_1,\ldots,G_m$ are complete & $G_1,\ldots,G_m$ are not  complete \\
  
\hline
$N_+(\varepsilon(H))$ & $N_+\big(\varepsilon(A(\overline{G_0}))\big)+ 1$ & $\sum_{i=0}^m N_+\big(\varepsilon(A(\overline{G_i}))\big) + 1 -m $\\
\hline
$N_0(\varepsilon(H))$ & $N_0\big(\varepsilon(A(\overline{G_0}))\big) + m(n-1)$ & $\sum_{i=0}^m N_0\big(\varepsilon(A(\overline{G_i}))\big)$\\
\hline
 $N_-(\varepsilon(H))$ & $N_-\big(\varepsilon(A(\overline{G_0}))\big) + m-1$ & $\sum_{i=0}^m N_-\big(\varepsilon(A(\overline{G_i}))\big)+m-1$\\
\hline
$\rho(H)$ & $2n(m-1)$ & $2(k+n(m-1))$ \\
\hline
$E(\varepsilon(H))$& $E\big({\varepsilon}(2A(\overline{G_0}))\big)  + 2\rho(H) $ & 
$\sum_{i=0}^m
E\big({\varepsilon}(2A(\overline{G_i}))\big)+ 2\rho(H) - 4km$\\
\hline
\end{tabular}
\end{table}
%$N_+(\varepsilon(H)) = \displaystyle \sum_{i=0}^m N_+\big(\varepsilon(A(\overline{G_i})\big) -m + 1$, \\
%$N_0(\varepsilon(H)) = \displaystyle \sum_{i=0}^m N_-0\big(\varepsilon(A(\overline{G_i})\big)$,\\ $N_{-}(\varepsilon(H)) = \displaystyle \sum_{i=0}^m N_-\big(\varepsilon(A(\overline{G_i})\big) + m-1$ and its eccentricity energy  is \\ 
%$E(\varepsilon(H)) = \displaystyle \bigg( \sum_{i=1}^m E\big(2A(\overline{G_i})\big)- 2mk \bigg) + 2(2(k+(m-1)n) - km)$.
\end{enumerate}
%$spec_\varepsilon(H) = \begin{Bmatrix}\frac{b+\sqrt{b^2+ 4kn}}{2}  & 0 & \frac{b -\sqrt{b^2+ 4kn}}{2} & r-2n \\1 & k(n-1) & 1 & k-1\end{Bmatrix}$ and $E_{\varepsilon}(H) = b+ |r-2n|(k-1)$, where $b = r+ 2n(k-1)$.     
\end{theorem}
%\newpage
\begin{proof} The proof depends when $G_0$ is complete or not. So, we have the following two cases. \\
\noindent\textbf{Case 1.} $G_0$ is  complete.\\
By Theorem \ref{BKP1006}(i), let $n_0 =\ell$, $n_i =n$ and  $k_i = k$, for all $i= 1,2,\ldots,m$.
Hence, 
$$\phi(\lambda, \varepsilon(H)) = \bigg(\displaystyle\prod_{i=1}^m  \frac{\phi(\lambda, 2A(\overline{G_i}))}{(\lambda-2k)} \bigg) \lambda^{\ell-1}
\big[(\lambda -2(k-n))^{m-1}(\lambda^2 -2(k+n(m-1))\lambda -\ell mn)\big].$$

Thus $\lambda_1=b + \sqrt{b^2 + \ell mn}$ is a positive eigenvalue of $\varepsilon(H)$ with multiplicity 1,  $\lambda_2 = b -\sqrt{b^2 + \ell mn}$, where $b =k+n(m-1)$ and $\lambda_3=2(k-n)$ are the 
negative eigenvalues of $\varepsilon(H)$ with multiplicity 1 and $m-1$, respectively.  
%are $\lambda_1 = 2(k-n)$ with multiplicity $m-1,\ \lambda_2 = b -\sqrt{b^2 + \ell mn}$
 %and a positive eigenvalue $\lambda_3 = b + \sqrt{b^2 + \ell mn}$ with multiplicity $1$, where $b =(k+n(m-1))$.
So, the spectral radius of $\varepsilon(H)$ is  $\rho(\varepsilon(H))=\lambda_1$  and $\xi(\varepsilon(G)) = \lambda_3$ is the smallest eigenvalue of $\varepsilon(H)$ (because $k<n$).  
%As $2k$ is an eigenvalue of $2A(\overline{G_i})$, for all $1\leq i\leq m$ and $N_+\big(2A(\overline{G_i})\big)= N_+\big(A(\overline{G_i})\big)$, we have 
From the characteristic polynomial $\phi(\lambda, \varepsilon(H))$ of $\varepsilon(H)$, the results follow.   
%The eccentricity energy of $H$ is given by\begin{align*}
 %   E({\varepsilon}(H))= |\end{align*}$N_+(\varepsilon(H)) = 
%\displaystyle \sum_{i=1}^m %N_+\big(A(\overline{G_i})\big) -(m- 1)$,\\
%$N_0(\varepsilon(H)) =\displaystyle \sum_{i=1}^m$$N_0\big(A(\overline{G_i})\big) + \ell -1$, and   $N_{-}(\varepsilon(H)) = \displaystyle\sum_{i=1}^m N_-\big(A(\overline{G_i})\big) +m$. Hence  the eccentricity energy  of $H$ is 
%$E_{\varepsilon}(H) = \displaystyle \bigg( \sum_{i=1}^m E_{\varepsilon}\big(2A(\overline{G_i})\big)- 2mk \bigg) + 2(\rho(H) -mk)$. 

\noindent\textbf{Case 2.} $G_0$ is not complete but $k_0$-regular.\\
Put $n_i =n$ and  $k_i = k$, for all $i= 1,2,\ldots,m$ in Theorem \ref{BKP1006}(ii), we see that\\ $\phi(\lambda I, \varepsilon(H)) =  \displaystyle  \frac{\prod_{i=0}^m\phi(\lambda, 2A(\overline{G_i}))}{(\lambda-2k)^m} \  \big\{\lambda - 2(k+n(m-1))\big\}\big\{\lambda -2(k-n)\big\}^{m-1}$.

Thus $\lambda_1= 2(k+n(m-1))$ is a positive eigenvalue of $\varepsilon(H)$ with multiplicity 1,  $\lambda_2 =2(k-n)$ is the 
negative eigenvalue of $\varepsilon(H)$ with multiplicity $m-1$. So, $\rho(H) = 2(k+n(m-1))$. Hence the results follow.   
%From the characteristic polynomial $\phi(\lambda, \varepsilon(H))$ of $\varepsilon(H)$, we see that
%$N_+(\varepsilon(H)) = \displaystyle \sum_{i=0}^m N_+\big(A(\overline{G_i})\big) -m + 1$, $N_0(\varepsilon(H)) = \displaystyle \sum_{i=0}^m N_0\big(A(\overline{G_i})\big)$,\\ $N_{-}(\varepsilon(H)) = \displaystyle \sum_{i=0}^m N_-\big(A(\overline{G_i})\big) + m-1$ and the eccentricity energy of $\varepsilon(H)$ is \\ $E_{\varepsilon}(H) = \displaystyle \bigg( \sum_{i=1}^m E{\varepsilon}\big(2A(\overline{G_i})\big)- 2mk \bigg) + 2(2(k+(m-1)n) - km)$. 
\end{proof}

As an immediate consequence of Theorem \ref{BKP1007}, we deduce the following results in \cite{Mahato1, Patel}.
\begin{corollary}[Theorem 3.1, \cite{Mahato1}]
Let  $K_{1,m}$ be the star graph on $m+1$ vertices. Then $Spec({\varepsilon}(K_{1,m})) = \begin{Bmatrix}
(m-1)\pm \sqrt{(m+1)^2 -3(m+1) + 3} & -2 \\
1 & m-1
\end{Bmatrix}$, and $\det(K_{1,m})= (-1)^mm2^{m-1}$. Further, if $m \geq 2$, then the least eigenvalue of $\varepsilon(K_{1,m})$ is $-2$.
\end{corollary}
\begin{proof}
If $G_i= K_1$, for $0 \leq i \leq m$ in Theorem \ref{BKP1007}(i), we get the result.   \end{proof}

\begin{corollary}[Theorem 3.3, \cite{Patel}]
The eccentricity spectrum of $W_{n+1}^{(m)}$ is
$Spec(\varepsilon(W_{n+1}^{(m)}))=\begin{Bmatrix}
b+\sqrt{b^2+mn} & 0 & b- \sqrt{b^2+mn} & -2n\\
1 & m(n-1) & 1 & m-1
\end{Bmatrix}$ and \\$E({\varepsilon}(W_{n+1}^{(m)})) = 2(b+ \sqrt{b^2 + mn)} = 2\rho(W_{n+1}^{(m)})$, where $b = n(m-1)$.
\end{corollary}
\begin{proof}
As $W_{n+1}^{(m)}\cong K_{1,m}[K_1,K_{n_1},\ldots,K_{n_m}]$, where $n = n_i, 1\leq i \leq m$ and by Theorem \ref{BKP1007}(i), the result follows.  
\end{proof}

\begin{corollary}[Corollary 3.3.1, \cite{Patel}]
Let $W_{n+1}^{(m)}$ be the windmill graph. Then \\
(i) $N_+(\varepsilon(\overline{H})) = 1$.
(ii) $N_0(\varepsilon(\overline{H}))=m(n-1)$.
(iii)$N_{-}(\varepsilon(\overline{H}))= m$.
\end{corollary}
\begin{proof}
By Theorem \ref{BKP1007}(i), the result follows. 
\end{proof}

\section*{Funding} 
The first author received support from the CSIR-UGC Junior Research Fellowship (UGC Ref. No.: 1085/(CSIR-UGC NET JUNE 2019)). The second author, support came from the National Board of Higher Mathematics (NBHM), Government of India, under grant No. 02011/50/2023 NBHM dated October 19, 2023.

\section*{Data availability} No data were used for the research described in the article.

\section*{Declarations} 

\noindent\textbf{Conflict of interest.} The authors declare that they have no conflict of interest.

\end{document}